\documentclass[11pt,letterpaper]{amsart}

\usepackage{amsmath}
\usepackage{graphicx}
\usepackage{amsmath,amssymb,amsfonts,amsthm,graphicx,amsopn,mathrsfs}
\usepackage[shortlabels]{enumitem}
\usepackage{tikz-cd}
\usepackage{tikz}
\usepackage{stmaryrd,url}
\usepackage{mathdots, verbatim,mathtools}
\usepackage{array,float}
\usepackage{fullpage}
\usepackage[utf8]{inputenc}
\usepackage{multirow}
\usepackage[normalem]{ulem}
\usepackage{hyperref}

\tikzstyle{nodal}=[circle,draw,fill=black,inner sep=0pt, minimum width=4pt]
\tikzstyle{half-fiber}=[rectangle,draw=black,thick,inner sep=0pt, minimum width=5pt, minimum height=5pt]
\tikzset{double distance = 2pt}

\newcommand{\vect}[1]{\boldsymbol{#1}}

\DeclareMathOperator{\Hom}{Hom}
\DeclareMathOperator{\GL}{GL}
\DeclareMathOperator{\Ker}{Ker}

\DeclareMathOperator{\rk}{rk}
\DeclareMathOperator{\Aut}{Aut}
\DeclareMathOperator{\id}{id}

\DeclareMathOperator{\disc}{disc}

\DeclareMathOperator{\Pic}{Pic}

\DeclareMathOperator{\MW}{MW}
\DeclareMathOperator{\Or}{O}

\usepackage[sorting=nyc,doi=false,isbn=false,url=false,giveninits=true,maxbibnames=99]{biblatex}

\DeclareSortingTemplate{nyc}{
  \sort{
    \field{presort}
  }
  \sort[final]{
    \field{sortkey}
  }
  \sort{
    \field{sortname}
    \field{author}
    \field{editor}
    \field{translator}
    \field{sorttitle}
    \field{title}
  }
  \sort{
    \field{sortyear}
    \field{year}
  }
  \sort{\citeorder}
}

\AtEveryBibitem{%
\ifentrytype{article}{
    \clearfield{number}
}{}
}
\renewbibmacro{in:}{%
  \ifentrytype{article}{}{\printtext{\bibstring{in}\intitlepunct}}}
\AtEveryBibitem{%
\ifentrytype{book}{
    \clearfield{pages}
}{}
}

\title{Borcherds lattices and K3 surfaces of zero entropy}
\author{Simon Brandhorst}
\address{Simon Brandhorst:
Fakult\"at f\"ur Mathematik und Informatik, Universit\"at des Saarlandes, Campus E2.4, 66123 Saarbr\"ucken, Germany}
\email{brandhorst@math.uni-sb.de}

\author{Giacomo Mezzedimi}
\address{Giacomo Mezzedimi:
Mathematisches Institut, Universit\"at Bonn, Endenicher Allee 60, 53115 Bonn, Germany.}
\email{mezzedim@math.uni-bonn.de}

\thanks{G.M. is funded by the Hausdorff Center for Mathematics, Bonn (Germany’s Excellence Strategy – EXC-2047/1–390685813).
S.B. is funded by the Deutsche Forschungsgemeinschaft (DFG, German Research Foundation) – Project-ID 286237555 – TRR 195.
Gefördert durch die Deutsche Forschungsgemeinschaft (DFG) – Projektnummer 286237555 – TRR 195.}

\date{\today}

\begin{document}

\newcommand{\QQ}{\mathbb{Q}}
\newcommand{\RR}{\mathbb{R}}
\newcommand{\ZZ}{\mathbb{Z}}
\newcommand{\CC}{\mathbb{C}}
\newcommand{\NN}{\mathbb{N}}
\newcommand{\PP}{\mathbb{P}}
\newcommand{\Stabe}{\Aut(\mathcal{D}_L,\vect{e})}
\newcommand{\Stabee}{\Aut(\mathcal{D}_L,\vect{e'})}

\newtheorem{theorem}{Theorem}[section]
\newtheorem{proposition}[theorem]{Proposition}
\newtheorem{lemma}[theorem]{Lemma}
\newtheorem{corollary}[theorem]{Corollary}
\newtheorem{claim}[theorem]{Claim}
\theoremstyle{definition}
\newtheorem*{acknowledgments}{\textbf{Acknowledgments}}
\newtheorem*{conv}{\textbf{Conventions}}
\newtheorem{definition}[theorem]{Definition}
\newtheorem{algorithm}[theorem]{Algorithm}
\theoremstyle{remark}
\newtheorem{remark}[theorem]{Remark}
\newtheorem{question}[theorem]{Question}
\newtheorem{example}[theorem]{Example}

\maketitle

\begin{abstract}
    Let $L$ be an even, hyperbolic lattice with infinitely many simple $(-2)$-roots. We call $L$ a Borcherds lattice if 
    it admits an isotropic vector with bounded inner product with all the simple $(-2)$-roots. 
    We show that this is the case if and only if $L$ has zero entropy, or equivalently if and only if all symmetries of $L$ preserve some isotropic vector. 
    
    We obtain a complete classification of Borcherds lattices, consisting of $194$ lattices. In turn this provides a classification of hyperbolic lattices of rank $\ge 5$ with virtually solvable symmetry group. Finally, we apply these general results to the case of K3 surfaces. We obtain a classification of Picard lattices of K3 surfaces of zero entropy and infinite automorphism group, consisting of $193$ lattices. In particular we show that all Kummer surfaces, all supersingular K3 surfaces and all K3 surfaces covering an Enriques surface (with one exception) admit an automorphism of positive entropy.
\end{abstract}

\section{Introduction}
\subsection{Borcherds lattices}
As observed by Conway, the Leech lattice $\Lambda$ has a striking property. Namely, the hyperbolic lattice $\mathrm{II}_{1,25}=U\oplus \Lambda$ admits an isotropic vector whose inner product with all the simple roots of $\mathrm{II}_{1,25}$ is bounded (more precisely, it is always $1$). Later Borcherds \cite{borcherds.leech} wondered which other hyperbolic lattices $L$ share this property with $\mathrm{II}_{1,25}$. He conjectured that $\mathrm{II}_{1,25}$ should be the lattice of maximal rank satisfying this property, and he asked for a classification.

Given our primary interest towards geometric applications, we concentrate on the $(-2)$-roots.
\begin{definition} \label{defn:intro}
 A \emph{Borcherds lattice} is an even hyperbolic lattice with infinitely many simple $(-2)$-roots which admits an isotropic vector with bounded inner product with all the simple $(-2)$-roots. 
\end{definition}
%(see \ref{sect:fundamental domain} for the definition of a simple $(-2)$-root). If $L$ admits no $(-2)$-root at all, then the previous condition is trivially satisfied; hence we add the requirement that a Borcherds lattice should contain a $(-2)$-root.
Notice that, since the set of $(-2)$-roots is a subset of the set of all roots, Borcherds lattices also satisfy Borcherds' original condition. Moreover the assumption that a Borcherds lattice should contain infinitely many simple $(-2)$-roots is not really restrictive: indeed if there are no $(-2)$-roots, the condition in Definition \ref{defn:intro} becomes vacuous, and we already have a classification of hyperbolic lattices with only finitely many simple $(-2)$-roots thanks to the work of Nikulin \cite{nikulin.finite.aut.3,nikulin.finite.aut.greater.5} and Vinberg \cite{vinberg.finite.aut.4}.

%Moreover, following Borcherds' terminology, we say that a negative definite lattice $W$ is a \emph{Leech type lattice} if $U\oplus W$ is a Borcherds lattice.
The main result of this paper is a classification of Borcherds lattices, as stated in the following theorem:

\begin{theorem} \label{thm:main.intro}
There are $194$ Borcherds lattices up to isometry. The maximum rank of a Borcherds lattice is $26$, achieved by the lattice $\mathrm{II}_{1,25}=U\oplus \Lambda$.
\end{theorem}
The interested reader can find the complete list in the ancillary file (see the Appendix for the list in rank $\ge 11$).
Theorem \ref{thm:main.intro} has profound algebraic and geometric implications, concerning isometry groups of lattices, automorphism groups of K3 surfaces and their discrete dynamics. In order to present these consequences, let us review how the theories of hyperbolic lattices and of K3 surfaces are closely intertwined.

To a K3 surface $X$ we can associate its Picard lattice $\Pic(X)$, which is an even hyperbolic lattice. The Picard lattice $\Pic(X)$ encodes not only a precise characterization of the smooth rational curves (they correspond to the simple $(-2)$-roots) and linear systems on $X$, but also the structure of the automorphism group $\Aut(X)$. Indeed $\Aut(X)$ coincides up to a finite group with the quotient $\Or^+(\Pic(X))/W(\Pic(X))$ of isometries of $\Pic(X)$ up to $(-2)$-reflections. Geometrically, this quotient can be identified with the group of isometries of $\Pic(X)$ preserving the nef cone of $X$, but it has the advantage of being a completely lattice-theoretical object: the symmetry group of $\Pic(X)$.

% In algebraic geometry the importance of hyperbolic lattices stems from the fact that several integral cohomology groups of projective varieties carry the structure of a lattice, and one can often obtain deep results about the geometry of a variety by means of linear algebra by looking at its associated lattices.
% %For instance, the N\'eron-Severi group modulo torsion of a smooth projective surface, equipped with the intersection product, is a hyperbolic lattice by the Hodge index theorem, and the same happens for the second integral cohomology of a hyperkähler variety.
% The most striking example is perhaps given by K3 surfaces: the Picard lattice $\Pic(X)$ of a K3 surface $X$ 
% %(which coincides with the N\'eron-Severi lattice $\mathrm{NS}(X)$), 
% encodes not only a precise characterization of the smooth rational curves and linear systems on $X$, but also the structure of the automorphism group $\Aut(X)$. Indeed $\Aut(X)$ coincides up to a finite group with the quotient $\Or^+(\Pic(X))/W(\Pic(X))$ of isometries of $\Pic(X)$ up to $(-2)$-reflections. Geometrically, this quotient can be identified with the group of isometries of $\Pic(X)$ preserving the nef cone of $X$, but it has the advantage of being a completely lattice-theoretical object: the symmetry group of $L$.

\subsection{Symmetries}
For a hyperbolic lattice $L$, denote by $\mathcal{D}_L$ the closure of a fundamental domain for the action of the Weyl group on the positive cone of $L$. It is a locally convex polyhedral cone whose walls correspond to the simple $(-2)$-roots.
We call 
\[\Aut(\mathcal{D}_L)\cong \Or^+(L)/W(L)\] 
the \emph{symmetry group} of $L$. The structure of the symmetry group and the geometry of the fundamental domain $\mathcal{D}_L$ are closely related to $L$ being a Borcherds lattice or not.  

To state our result, we need two more definitions.
A primitive isotropic vector in $\mathcal{D}_L$ is called a cusp.
Furthermore, we say that a symmetry $f\in \Aut(\mathcal{D}_L)$ has \emph{zero entropy} if its spectral radius is at most one, and $L$ has \emph{zero entropy} if every symmetry of $L$ has zero entropy.

\begin{theorem} \label{thm:equivalence.intro}
Let $L$ be an even hyperbolic lattice.
The following are equivalent:
\begin{enumerate}[(a)]
    \item $L$ is a Borcherds lattice;
    \item $L$ has an infinite symmetry group and zero entropy;
    \item $\Aut(\mathcal{D}_L)$ is infinite and it preserves a unique cusp.
    \item $\mathcal{D}_L$ has a unique cusp with infinite stabilizer.
\end{enumerate}
\end{theorem}
The equivalence of (b), (c) and (d) is known. See the work of Nikulin  \cite[Theorem~9.1.1]{nikulin.factor.groups} and Oguiso \cite[Theorem~1.4]{oguiso.entropy}; we give an alternative proof by means of hyperbolic geometry.
We refer to Theorem \ref{thm:equivalence.characterizations.Borcherds} for the complete statement with all the equivalent characterizations of Borcherds lattices.

The connection with Borcherds lattices and automorphisms of K3 surfaces leads us to a closer study of symmetry groups of hyperbolic lattices.
According to the Tits alternative \cite{detinko.flannery.obrien.tits} there are two options for $\Aut(\mathcal{D}_L)$. Either it is virtually solvable or it contains a free non-abelian subgroup, where we say the that a property of a group holds virtually if it holds for a finite index subgroup.

If the rank of $L$ is at most $2$, standard arguments show that the symmetry group of $L$ is either finite, or virtually abelian.
Moreover, thanks to the work of Nikulin \cite{nikulin.finite.aut.3,nikulin.finite.aut.greater.5} and Vinberg \cite{vinberg.finite.aut.4}, we already have a complete classification of hyperbolic lattices of rank $\ge 3$ with finite symmetry group, consisting of $118$ lattices. Therefore in the following we will restrict our attention to hyperbolic lattices with infinite symmetry group.

As a consequence of Theorem \ref{thm:equivalence.intro}, the symmetry group of a Borcherds lattice $L$ coincides with the stabilizer of a cusp of $\mathcal{D}_L$, and consequently it is virtually abelian by Proposition \ref{prop:generalization.Shioda.Tate}. Surprisingly a partial converse holds. It is a slight strengthening of a theorem of Nikulin \cite[Theorem~9.1.1]{nikulin.factor.groups}.

\begin{theorem}\label{thm:virtualy-solvable5}
Every hyperbolic lattice of rank at least $5$ with an infinite, virtually solvable symmetry group is a Borcherds lattice.
\end{theorem}
See Remark \ref{rk:useful.remarks.borcherds} for counterexamples in rank $\le 4$. As a corollary we obtain the  classification of hyperbolic lattices of rank at least $5$ and virtually solvable automorphism group.

\subsection{Consequences for K3 surfaces}
We work over an algebraically closed field $k=\bar k$ of arbitrary characteristic.
The (algebraic) entropy of an automorphism $f$ of a smooth projective surface $X$ is defined as the logarithm of the spectral radius of its induced action on the Picard lattice $\Pic(X)$ by pullback, and it is a nonnegative number. If $f$ is an automorphism of a smooth projective \emph{complex} surface $X$, it agrees with the topological entropy of $f$, which measures how fast points spread apart under the iteration of $f$.
Nevertheless, the entropy is a measure for its dynamical complexity in any characteristic.
Indeed $f$ is of zero entropy if and only if either it is of finite order or it preserves an elliptic fibration \cite[Thm. 2.11,\S 2.4.5]{cantat:dynamics_of_compact_complex}. The class of a fiber in $\Pic(X)$ is a cusp of the nef cone.

K3 surfaces are one of the few classes of surfaces, together with rational, abelian and Enriques surfaces, that can admit automorphisms of positive entropy \cite{cantat.dynamics.K3.1,cantat.dynamics.K3.2}. It is therefore relevant, from a dynamical standpoint, to understand which K3 surfaces admit an automorphism of positive entropy. 
We completely answer this question by providing an exhaustive classification of K3 surfaces of \emph{zero entropy}, i.e. K3 surfaces admitting only automorphisms of zero entropy. In fact combining Theorem \ref{thm:main.intro} and Theorem \ref{thm:equivalence.intro} we obtain:

\begin{corollary} \label{cor:main.intro.K3}
A K3 surface $X$ has zero entropy and an infinite automorphism group if and only if its Picard lattice is a Borcherds lattice, or equivalently if $\Pic(X)$ is one of the lattices classified in Theorem \ref{thm:main.intro}.
\end{corollary}

In view of Theorem \ref{thm:virtualy-solvable5}, this includes a classification of K3 surfaces with virtually solvable automorphism group and Picard rank at least 5.

By the surjectivity of the period map, a hyperbolic lattice $L$ is the Picard lattice of some complex K3 surface if and only if $L$ embeds primitively into the K3 lattice $U^3\oplus E_8^2$. Quite remarkably, we observe that all Borcherds lattices, with the obvious exception of $\mathrm{II}_{1,25}$, embed primitively into the K3 lattice.
This leads to a list of exactly $193$ families of complex K3 surfaces whose very general member has zero entropy.

Let us state some important consequences of Corollary \ref{cor:main.intro.K3}, which we collect in the following corollary.

\begin{corollary} \label{cor:main.consequences}
The following K3 surfaces admit an automorphism of positive entropy, and in particular their automorphism group is not virtually solvable: 
\begin{itemize}
    \item Kummer surfaces in characteristic $0$ or $p>2$;
    \item K3 surfaces covering an Enriques surface, unless $\Pic(X)\cong U\oplus E_8\oplus D_8$;
    \item Singular and supersingular K3 surfaces;
    \item K3 surfaces with an elliptic fibration of Mordell-Weil rank at least 9.
\end{itemize}
\end{corollary}

We refer the reader to Remark \ref{rk:U.E8.D8} for a detailed explanation of the geometry of K3 surfaces with Picard lattice $U\oplus E_8\oplus D_8$.

The problem of determining the list of hyperbolic lattices of zero entropy has a long history. Nikulin showed in \cite{nikulin.finite.aut.greater.5} that several $2$-elementary Picard lattices of K3 surfaces have zero entropy, and he obtained a partial classification of K3 surfaces of zero entropy and Picard rank $3$ in \cite[Theorem~3~and~the~subsequent~discussion]{nikulin.interesting}. On the other hand, Oguiso \cite[Theorem~1.6]{oguiso.entropy} showed that every singular K3 surface has positive entropy.
More recently, the second author obtained in \cite[Theorem~6.12]{mezzedimi.entropy} a classification of Picard lattices of K3 surfaces of zero entropy admitting an elliptic fibration with only irreducible fibers: the list comprises of $32$ lattices. Moreover he showed that every K3 surface with infinite automorphism group and Picard rank $\ge 19$ has positive entropy, extending Oguiso's result.

The classification of K3 surfaces of zero entropy is independently obtained by Yu \cite[Theorem~1.1]{yu.entropy} in his recent preprint. We note that our classification in Corollary \ref{cor:main.intro.K3} agrees with Yu's.

\subsection{Strategy}
The first step towards Theorem \ref{thm:main.intro} is the classification of \emph{Leech type lattices}, i.e. negative definite lattices $W$ such that $L=U\oplus W$ is a Borcherds lattice.
%Assume that $L=U\oplus W$ for a certain negative definite lattice $W$. 
According to Proposition \ref{prop:genus.Leech.type.lattice}, if $W$ is a Leech type lattice, then the genus of $W$ contains precisely one lattice that is not an overlattice of a root lattice. This naturally divides our work into two parts: when $W$ is an overlattice of a root lattice, and when $W$ is unique in its genus. Our strategy is to first reduce to a finite problem, by excluding all but finitely many negative definite lattices, and then checking whether the remaining ones are of Leech type. In order to decide whether a certain hyperbolic lattice is a Borcherds lattice, we compute its symmetry group via Borcherds' method, and we check whether the symmetry group preserves an isotropic vector.
% In the case when $W$ is an overlattice of a root lattice, we use the fact that there are only finitely many overlattices of root lattices in each rank, and therefore it suffices to show that lattices of rank $\ge 25$ cannot be of Leech type (see Proposition \ref{prop:no.leech.rank.25}).
% On the other hand, we have a complete and explicit classification of definite lattices (of rank at least 3) unique in their genus.
% %, originally due to Watson and later completed and corrected by Voight \cite{voight}, Lorch and Kirschmer \cite{lorch.kirschmer.single.class}. 
% This list consists of finitely many lattices, and all of their multiples. We employ the same strategy as in \cite[Theorem~4.6]{mezzedimi.entropy} to find an effective bound on the number of multiples of a given lattice that are of Leech type, and again this is sufficient to produce a finite list of candidates.

The second and final step towards the classification in Theorem \ref{thm:main.intro} concerns hyperbolic lattices that do not contain a copy of $U$. In order to deal with this case, we show in Proposition \ref{prop:Gram.matrix.Borcherds.lattices} that every Borcherds lattice $L$ is a sublattice of ``small'' index of a second Borcherds lattice $L'=U\oplus W$ containing a copy of $U$. Since we already have a complete classification of Leech type lattices, this is enough to produce a finite list of candidate lattices, which again we study individually to decide whether they are Borcherds lattices. 

\subsection{Outline}
In Section \ref{sec:lattices} we recall some well-known properties of negative definite and hyperbolic lattices. In Section \ref{sec:Borcherds.lattices} we introduce Borcherds lattices and we prove Theorem \ref{thm:equivalence.intro}, showing several equivalent characterizations of Borcherds lattices. 
Section \ref{sec:classification} is the core of the article: we obtain first the classification of Leech type lattices in Theorem \ref{thm:leech.classification}, and then the classification of Borcherds lattices in Theorem \ref{thm:classification.Borcherds.lattices}. Finally, we devote Section \ref{sec:K3} to applications of our classification to K3 surfaces. We obtain the classification of K3 surfaces of zero entropy as in Corollary \ref{cor:main.intro.K3} and we deduce Corollary \ref{cor:main.consequences}.
See the appendix \ref{sec:borcherds.method} for a brief review of Borcherds' method as well as improvements to the algorithm.
Our implementation of Borcherds' method is part of the computer algebra system OSCAR \cite{OSCAR}. 

\subsection*{Acknowledgements}
We thank Gebhard Martin for several helpful discussions. The second author wishes to thank the University of Hannover, where part of this work was carried out, for the stimulating environment during his time as PhD student.

\section{Preliminaries on lattices} \label{sec:lattices}

In this section we recall the basics of lattices, with particular emphasis towards negative definite and hyperbolic even lattices. The main references are \cite{nikulin.integral.symmetric.bilinear}, \cite{conway.sloane.lattices} and \cite{ebeling.lattices}.

\subsection{Basic definitions} 

A \emph{lattice} is a finitely generated abelian group $L$ endowed with a symmetric, nondegenerate, integral bilinear form. 
$L$ is \emph{even} if the square of every vector of $L$ is an even number, otherwise it is \emph{odd}. We will be mainly interested in even lattices, so in the following every lattice will be even, unless otherwise specified.

The \emph{rank} $\rk(L)$ of $L$ is its rank as an abelian group, and the \emph{discriminant} $\disc(L)$ is the absolute value of the determinant of the Gram matrix of $L$ with respect to any basis. A lattice is called \emph{unimodular} if it has discriminant $1$.
The \emph{signature} $(l_+,l_-)$ of $L$ is the signature of the real bilinear form on the real vector space $L\otimes \RR$. We say that $L$ is \emph{positive} (resp. \emph{negative}) \emph{definite} if its signature is $(\rk(L),0)$ (resp. $(0,\rk(L))$), and \emph{hyperbolic} if its signature is $(1,\rk(L)-1)$.
We denote by $U$ the hyperbolic plane, i.e. the only even, unimodular, hyperbolic lattice of rank $2$.

For a lattice $L$ and an integer $n\ne 0$, we will denote $L(n)$ the lattice with the bilinear form of $L$ multiplied by $n$. In particular a lattice $L$ is positive definite if and only if $L(-1)$ is negative definite. If $n>0$, we will refer to the lattices $L(n)$ as the \emph{multiples} of $L$.

The \emph{dual lattice} $L$ is defined as $L^\vee=\{\vect{v}\in L\otimes \QQ : \vect{v}.L\subseteq \ZZ\}$, together with the natural extension of the bilinear form on $L$. The \emph{discriminant group} of the even lattice $L$ is the finite group $A_L=L^\vee/L$, together with the finite quadratic form with values in $\QQ/2\ZZ$ defined by $\overline{\vect{v}}.\overline{\vect{v}}=\vect{v}.\vect{v} \pmod{2\ZZ}$, where $\overline{\vect{v}}$ denotes the class in $A_L$ of $\vect{v}\in L^\vee$.
The cardinality of $A_L$ coincides with the discriminant of the lattice $L$. The \emph{length} $\ell(A_L)$ is defined as the minimal number of generators of the abelian group $A_L$, and clearly $\ell(A_L)\le \rk(L)$.

\subsection{Overlattices}

Given a lattice $L$, we say that $M$ is an \emph{overlattice} of $L$ if $M$ contains $L$ and the index $[M:L]$ as abelian groups is finite. In particular the overlattices of $L$ have the same signature of $L$.

Recall that the overlattices of a given lattice $L$ correspond to isotropic subgroups of the discriminant group $A_L$ {\cite[Proposition~1.4.1]{nikulin.integral.symmetric.bilinear}}.

\subsection{Root lattices and root overlattices}

If $L$ is a lattice, a \emph{$(-2)$-root} is a vector of $L$ of square $-2$. We denote the set of $(-2)$-roots of $L$ by $\Delta_L$.
The sublattice $L_{root}$ of $L$ spanned by the $(-2)$-roots is called the \emph{root part} of $L$.

A negative definite lattice $R$ is called a \emph{root lattice} if it admits a generating set of $(-2)$-roots, i.e. if $R_{root}=R$ holds. Any root lattice can be decomposed as a direct sum of \emph{ADE lattices}, i.e. of the lattices $A_n$, $D_n$ (for $n\ge 4$) and $E_n$ (for $6\le n \le 8$) \cite[Theorem~1.2]{ebeling.lattices}. ADE lattices correspond to (simply laced) Dynkin diagrams. In particular there are only finitely many root lattices of rank $r$ up to isometry.

A \emph{root overlattice} $W$ is a negative definite lattice that is an overlattice of a root lattice, or equivalently such that $\rk(W_{root})=\rk(W)$. Since the overlattices of a root lattice $R$ correspond to certain subgroups of the finite discriminant group $A_R$, we obtain that there are only finitely many root overlattices of rank $r$ up to isometry.

\subsection{Genus of a lattice}

Two lattices $L$ and $M$ with the same signature are \emph{in the same genus} if $A_L\cong A_M$ as finite quadratic spaces, i.e. there exists an isomorphism of groups $A_L\cong A_M$ preserving the quadratic forms. 
Equivalently, $L$ and $M$ are in the same genus if and only if $U\oplus L$ and $U\oplus M$ are isometric (cf. {\cite[Corollary~1.13.4]{nikulin.integral.symmetric.bilinear}}).

The \emph{genus} of $L$ is the list of all lattices in the genus of $L$, considered up to isometry.

While many indefinite lattices are unique in their genus (for instance, all lattices of the form $U\oplus L$, cf. {\cite[Corollary~1.13.3]{nikulin.integral.symmetric.bilinear}}), most definite lattices are not. In a series of papers Watson produced by hand the finite (up to multiples and isometry) list of positive definite lattice of rank at least $3$ which are unique in their genus. Later Watson's results were completed, corrected and extended with computer aid by Lorch and Kirschmer \cite{lorch.kirschmer.single.class} and Voight \cite{voight}. See the catalogue of lattices \cite{catalogue} for the list. Unfortunately, the classification in rank $2$ is still conditional on the Generalized Riemann Hypothesis (GRH). We explain in Section \ref{sec:GRH} how we bypass the classification in rank $2$ in order to make our results independent of the GRH.

\subsection{Primitive embeddings}

An \emph{embedding} $i:L\hookrightarrow M$ of lattices is an injective homomorphism that preserves the bilinear products. The embedding $i:L\hookrightarrow M$ is said to be \emph{primitive} if the cokernel $M/i(L)$ is torsion free. If it is not primitive, its \emph{saturation} is the smallest primitive sublattice of $M$ containing the image $i(L)$.

\subsection{Fundamental domain of a hyperbolic lattice}\label{sect:fundamental domain}
Our account follows \cite[Chapter 8, \S 2]{huybrechts.K3}.
In this section $L$ will always denote a hyperbolic lattice.
The positive cone $\mathcal{P}_L$ of $L$ is a fixed connected component of $\{\vect{x}\in L\otimes \RR : \vect{x}^2>0\}$.
Let $\mathbb{H}_L$ be the sheet of the hyperboloid $\{\vect{x}\in L\otimes \RR : \vect{x}^2 = 1\}$ contained in $\mathcal{P}_L$.
Since the signature of $L\otimes \RR$ is $(1,\rk(L)-1)$, $\mathbb{H}_L$ is a model for the hyperbolic space of dimension $\rk(L)-1$. We denote by $\overline{\mathbb{H}}_L$
the \emph{conformal ball model} of hyperbolic space, cf. \cite[§4.5]{ratcliffe}. The boundary points $\partial \mathbb{H}_L$ correspond to the isotropic rays of $L$. Note that any isometry of $\mathbb{H}_L$ extends to $\overline{\mathbb{H}}_L$.
We will denote by $\Or^+(L)$ the group of isometries of $L$ preserving $\mathcal{P}_L$.

A $(-2)$-root $\vect{r}$ in $L$ induces the \emph{reflection} $s_{\vect{r}}\in \Or^+(L)$ into the mirror $\vect{r}^\perp$ such that $s_{\vect{r}}(\vect{v})=\vect{v}+(\vect{v}.\vect{r})\vect{r}$ for any $\vect{v}\in L$. %Notice that $s_{\vect{r}}$ is an isometry of order $2$. 
The \emph{Weyl group} of $L$ is the subgroup $W(L)$ of $\Or^+(L)$ generated by the reflections in $(-2)$-roots of $L$. 
 The mirrors cut the positive cone into connected components, called \emph{chambers}, and the Weyl group acts simply transitively on the set of all chambers. 
 We will denote by $\mathcal{D}_L$ the closure (in $L_\RR$) of a fundamental domain for the action of $W(L)$ on the positive cone, i.e. $\mathcal{D}_L$ is the closure of a chamber.  

Since the Weyl group $W(L)$ is normal in $\Or^+(L)$, we can consider the quotient $\Or^+(L)/W(L)$ of isometries of $L$ up to reflections. By construction it can be identified with the group $\Aut(\mathcal{D}_L)$ of isometries of $L$ preserving $\mathcal{D}_L$, and it is called the \emph{symmetry group} of $L$.

A vector $\vect{v}\in L$ is said to be \emph{fundamental} if it belongs to the chosen fundamental domain $\mathcal{D}_L$. A \emph{cusp} of $\mathcal{D}_L$ is a primitive isotropic vector $\vect{e} \in L\cap \mathcal{D}_L$. In particular it is fundamental.

If $\vect{v}\ne \vect{0}$, $\vect{v}$ is \emph{positive} if it has nonnegative inner product with all the fundamental vectors. In particular, the elements of the positive cone are positive.
Moreover by definition any fundamental vector has nonnegative square. 

Note that the fundamental vectors are precisely those elements of $\overline{\mathcal{P}}_L$ intersecting all \emph{positive} $(-2)$-roots nonnegatively.
Moreover, for any $(-2)$-root $\vect{r}\in L$, either $\vect{r}$ or $-\vect{r}$ is positive. 
%Indeed, if neither were positive, there would be two vectors $\vect{h},\vect{h'}$ in the interior of the the fundamental domain such that $\vect{h}.\vect{r}<0$ and $\vect{h'}.(-\vect{r})<0$, i.e. $\vect{h'}.\vect{r}>0$. Since the map $\mathcal{D}_L^\circ \rightarrow \RR$ given by intersecting with $\vect{r}$ is continuous, there would be a vector in the interior $\mathcal{D}_L^\circ$ of the fundamental domain that is orthogonal to a $(-2)$-root, and this is impossible. The same holds for vectors of square $\ge 0$, since no vector in the interior of the fundamental domain is orthogonal to a vector of nonnegative square.

Finally, we will say that a positive $(-2)$-root $\vect{r}\in L$ is \emph{simple} if $\vect{r}-\vect{r'}$ is not positive for any positive $(-2)$-root $\vect{r'}\in L$ different from $\vect{r}$.
Note that $\vect{r}$ is simple if and only if $\vect{r}^\perp \cap \mathcal{D}_L$ is a facet, i.e. a codimension 1 face of $\mathcal{D}_L$. 
By definition, all positive $(-2)$-roots can be written as sums of simple $(-2)$-roots. 

We conclude the section with the following important result (note that even though it is stated only for Picard lattices of K3 surfaces, the proof is entirely lattice-theoretical):

\begin{proposition}[{\cite[Corollary~8.4.7]{huybrechts.K3}}] \label{prop:finitely.many.roots}
Let $L$ be a hyperbolic lattice containing at least one $(-2)$-root. Then the symmetry group $\Aut(\mathcal{D}_L)$ is finite if and only if $L$ contains finitely many simple $(-2)$-roots, i.e. if $\mathcal{D}_L$ is a finite polyhedral cone.
\end{proposition}

\section{Borcherds lattices} \label{sec:Borcherds.lattices}

In this section we introduce the main objects of the article, namely Borcherds lattices. After reviewing the definition of entropy of isometries of a hyperbolic lattice, we prove several equivalent characterizations of Borcherds lattices.

\subsection{Entropy on hyperbolic lattices}

Let $L$ be a hyperbolic lattice, and $g\in \Or^+(L)$ an isometry. The \emph{entropy} of $g$ is the nonnegative number $h(g)=\log{\lambda(g_\CC)}$, where $g_\CC$ is the natural extension of $g$ to $L\otimes \CC$ and $\lambda(g_\CC)$ is the spectral radius of $g_\CC$, i.e. the maximum norm of its eigenvalues.
Clearly isometries of finite order have zero entropy, since the eigenvalues of $g_\CC$ are roots of unity. The converse is not true, but we can characterize isometries of zero entropy by recalling the following classification of isometries of hyperbolic space (see \cite{ratcliffe} for more details).
If $g$ is an isometry of the hyperbolic space $\mathbb{H}^n$, we say that:
\begin{itemize}
    \item $g$ is \emph{elliptic} if $g$ preserves a point in $\mathbb{H}^n$;
    \item $g$ is \emph{parabolic} if it is not elliptic and it fixes a unique point in the boundary $\partial\mathbb{H}^n$;
    \item $g$ is \emph{hyperbolic} if it is not elliptic and it fixes two points in the boundary $\partial\mathbb{H}^n$.
\end{itemize}

Any isometry $g\in \Aut(\mathcal{D}_L)$ induces an isometry $g_\mathbb{H}$
of the hyperbolic space $\mathbb{H}_L$; hence we will say that $g$ is elliptic, parabolic or hyperbolic if $g_\mathbb{H}$ is so.

Elliptic isometries in $\Aut(\mathcal{D}_L)$ have finite order, since they are conjugate to rational orthogonal transformations of euclidean space \cite[Theorem~5.7.1]{ratcliffe}.
Parabolic isometries in $\Aut(\mathcal{D}_L)$ fix a unique point in the boundary of $\mathbb{H}_L$, hence they fix an isotropic ray in $\mathcal{D}_L$, which is generated by a primitive isotropic vector of $L$ by \cite[Remarque~1.1]{cantat.dynamics.K3.1}. 
A parabolic isometry $g$ is conjugate to a product $g_sg_u$, with $g_s$ elliptic and $g_u$ unipotent such that $g_s$ and $g_u$ commute \cite[Theorem~4.7.3]{ratcliffe}. In particular every eigenvalue of $g$ lies on the unit circle.
Since $g$ is defined over the rationals, Kronecker's theorem implies that each eigenvalue is a root of unity. 
On the other hand hyperbolic isometries in $\Aut(\mathcal{D}_L)$ fix two isotropic rays in $\mathcal{D}_L$, none of which is defined over $\QQ$ by \cite[Remarque~1.1]{cantat.dynamics.K3.1}. The eigenvalues of a hyperbolic isometry are $\{\lambda_1,\ldots,\lambda_r,\lambda,\lambda^{-1}\}$, where the $\lambda_i$ have absolute value $1$ and $\lambda$ is a \emph{Salem number} (cf. \cite[Discussion~before~Définition~1.2]{cantat.dynamics.K3.1}).

It is immediate to notice that elliptic and parabolic isometries have zero entropy, while hyperbolic isometries have positive entropy. From the previous discussion it follows immediately:

\begin{proposition} \label{prop:zero.entropy.isotropic}
An isometry $g\in \Aut(\mathcal{D}_L)$ has zero entropy if and only if either it has finite order, or if $g$ preserves a cusp of $\mathcal{D}_L$.
\end{proposition}

We will denote by $\Stabe$ the subgroup of $\Aut(\mathcal{D}_L)$ preserving the cusp $\vect{e}$.
By the Shioda-Tate formula, on K3 surfaces the size of the stabilizer of a nef isotropic vector $\vect{e}$ (which corresponds to the rank of the Mordell-Weil group of the Jacobian fibration of $|\vect{e}|$) depends on the rank of the root part $R=(\vect{e}^\perp /\langle \vect{e}\rangle)_{root}$ of $\vect{e}^\perp /\langle \vect{e}\rangle$. The following result, known to the experts, shows that the same happens on general hyperbolic lattices. For lack of a reference we give a proof. Recall that a group is \emph{virtually abelian} if it contains an abelian subgroup of finite index. 

\begin{proposition} \label{prop:generalization.Shioda.Tate}
Let be a cusp of $\mathcal{D}_L$ and $\Stabe$ its stabilizer. The group $\Stabe$ is virtually abelian, and more precisely it contains a normal subgroup $G$ of finite index isomorphic to $\ZZ^m$, where $m=\rk(L)-2-\rk(R)$ and $R$ is the root part of $\vect{e}^\perp /\langle \vect{e}\rangle$.

In particular $\Stabe$ is finite if and only if $\vect{e}^\perp / \vect{e}$ is a root overlattice.
\end{proposition}

\begin{proof}
Let $W \coloneqq \vect{e}^\perp / \langle\vect{e}\rangle$ and denote the stabilizer of $\vect{e}$ in $\Or(L)$ by $\Or(L,\vect{e})$. Let $\pi_{\vect{e}}:\Or(L,\vect{e})\rightarrow \Or(W)$ be the natural homomorphism.
We set $H \coloneqq \Ker(\pi_{\vect{e}})$ and $G = H \cap \Stabe$. Clearly $G$ is a normal subgroup of $\Stabe$ of finite index, since $W$ is negative definite and thus $\Or(W)$ is finite.

For any $h\in H$, we have by construction that $h(\vect{e})=\vect{e}$ and $h(\vect{w})=\vect{w}+\alpha_h(\vect{w})\vect{e}$ for every $\vect{w}\in \vect{e}^\perp$, where $\alpha_h(\vect{w}) \in \ZZ$. Note that for $k \in \ZZ$ we have $\alpha(\vect{w} +k \vect{e})=\alpha(\vect{e})$.
Thus we obtain a homomorphism 
\[\varphi\colon H \to \Hom(W,\ZZ)\cong W^\vee, \quad  h \mapsto \alpha_h,\]
where we identify $\Hom(W,\ZZ)$ and $W^\vee$ via the canonical isomorphism induced by the bilinear form. 

We prove that $\varphi$ is injective.
Let $h \in \Ker(\varphi)$ and choose $\vect{x}\in L$ with $\vect{x}.\vect{e}\neq 0$. After replacing $\vect{x}$ by $2(\vect{x}.\vect{e}) \vect{x} - \vect{x}^2 \vect{e}$ we can assume $\vect{x}^2=0$. Write $h(\vect{x}) = a \vect{x}+b\vect{e}+\vect{v}$ with $v \in V:=\langle \vect{x},\vect{e}\rangle^\perp$.
Since $h \in \Ker(\varphi)$, we have $h|_{\vect{e}^\perp}=\id$. The fact that $h$ is an isometry implies that $0 = \vect{x}.\vect{v}' = h(\vect{x}).\vect{v}'=\vect{v}.\vect{v}'$ for any $\vect{v}' \in V$. Thus $\vect{v}=0$. We have $\vect{x}.\vect{e}=h(\vect{x}).\vect{e} = a \vect{x}.\vect{e}$ and therefore $a=1$. 
Further $0=\vect{x}^2=h(\vect{x})^2=2b \vect{x}.\vect{e}$, implies $b=0$ and in turn $h=\id$. 
Let $R$ be the root part of $W$ and $R^\perp \subseteq W$ its orthogonal complement.
The proposition is proven if we can show that \[R^\perp \subseteq \varphi(G)\subseteq R^\perp\otimes \QQ.\]

First we show that $\varphi(G)\subseteq R^\perp\otimes \QQ$, i.e. any $g \in G$ acts trivially on $R$.
Notice that it is sufficient to show that $G$ fixes all the simple $(-2)$-roots of $L$ orthogonal to $\vect{e}$, because any root in $R$ is represented by a linear combination of simple $(-2)$-roots orthogonal to $\vect{e}$.
If $\vect{r}$ is a simple $(-2)$-root orthogonal to $\vect{e}$ and $g\in G$, then by construction of $G$ we have that $g(\vect{r})=\vect{r}+\beta \vect{e}$ for some $\beta \in \ZZ$. Since $g(\vect{r})$ is a simple $(-2)$-root as well,  $g(\vect{r})-\vect{r}=\beta \vect{e}$ cannot be positive, unless $\vect{r}=g(\vect{r})$. Therefore $\beta \leq 0$.
However, if $\beta < 0$, then $\vect{r}+\beta \vect{e}$ is not positive, since it intersects negatively any fundamental vector in $\vect{r}^\perp$ different from $\vect{e}$ (recall that $\vect{r}^\perp \cap \mathcal{D}_L$ is a facet). We conclude that $g$ fixes $\vect{r}$.
Therefore the image $\varphi(G)$ is contained in $R^\perp \otimes \QQ$.

Now we show $R^\perp\subseteq \varphi(G)$.
 For $\vect{y} \in \vect{e}^\perp$ with $\vect{y}+\ZZ \vect{e} \in R^\perp$, the Eichler-Siegel transformation $\psi_{\vect{y}} \in \Or(L)$ is defined by
 \[\psi_{\vect{y}}(\vect{x})=\vect{x}+(\vect{x}.\vect{y})\vect{e} - (\vect{x}.\vect{e})\vect{y}-\frac{1}{2}(\vect{x}.\vect{e})\vect{y}^2 \vect{e}.\]
 Clearly $\psi_{\vect{y}}$ acts trivially on $\langle \vect{e},\vect{y}\rangle^\perp$ and $\psi_{\vect{y}}(\vect{y})=\vect{y}+\vect{y}^2\vect{e}$. Thus $\psi_{\vect{y}}$ descends to the identity of $W$, i.e. $\psi_{\vect{y}} \in H$. 
 In order to show that $\psi_{\vect{y}} \in \Aut(\mathcal{D}_L,\vect{e})$, we prove that $\psi_{\vect{y}}$ preserves the set of positive roots of $L$.
 The isometry $\psi_{\vect{y}}$ acts trivially on $R$ and therefore preserves the simple roots perpendicular to $\vect{e}$. So let $\vect{r}\in L$ be a positive $(-2)$-root with $\vect{e}.\vect{r}>0$. The image $\psi_{\vect{y}}(\vect{r})$ is a $(-2)$-root, so it is either positive or $-\psi_{\vect{y}}(\vect{r})$ is positive. However $$\psi_{\vect{y}}(\vect{r}).\vect{e} = \psi_{\vect{y}}(\vect{r}).\psi_{\vect{y}}(\vect{e}) = \vect{r}.\vect{e}>0,$$
hence $\psi_{\vect{y}}(\vect{r})$ is positive.
We have shown that $\psi_{\vect{y}} \in G$. Moreover $\varphi(\psi_{\vect{y}})=\vect{y}+\ZZ\vect{e}$, hence $R^\perp \subseteq \varphi(G)$.
\end{proof}

For a subgroup $G$ of $\Aut(\mathcal{D}_L)$, we say that $G$ has \emph{zero entropy} if all the elements of $G$ have zero entropy, otherwise we say that $G$ has \emph{positive entropy}.

\begin{definition}
A hyperbolic lattice $L$ has \emph{zero entropy} if its symmetry group $\Aut(\mathcal{D}_L)$ has zero entropy. Otherwise we say that $L$ has \emph{positive entropy}.
\end{definition}

\begin{proposition} \label{prop:entropy.overlattice}
Every overlattice of a hyperbolic lattice of zero entropy has zero entropy as well.
\end{proposition}
\begin{proof}
Let $M$ be any overlattice of $L$.
Without loss of generality we can assume that the fundamental domain $\mathcal{D}_M$ of $M$ is contained in the fundamental domain $\mathcal{D}_L$ of $L$. Let $g \in \Aut(\mathcal{D}_M)$. Since $M/L$ is finite, there is $n\ge 1$ such that $g^n$ preserves $L$. Therefore $g^n \in \Aut(\mathcal{D}_L)$. By assumption $g^n$ has zero entropy, and therefore $g$ has zero entropy as well.
\end{proof}

\subsection{Borcherds and Leech type lattices}

In the following $L$ will denote a hyperbolic isotropic lattice, and $\mathcal{D}_L$ a fixed fundamental domain for the Weyl group. 
Recall that a cusp of $\mathcal{D}_L$ is a primitive isotropic vector $\vect{e} \in L\cap \mathcal{D}_L$.

\begin{definition} \label{definition:Borcherds.lattice}
An even hyperbolic lattice $L$ with infinitely many simple $(-2)$-roots is a \emph{Borcherds lattice} if there exists a cusp $\vect{e}$ of the fundamental domain $\mathcal{D}_L$  having bounded inner product with all the simple $(-2)$-roots of $L$.
\end{definition}

We concentrate on even lattices and $(-2)$-roots in view of our geometric applications.
The core of Definition \ref{definition:Borcherds.lattice} is the existence of an isotropic vector with bounded inner product with all the simple $(-2)$-roots. 
Note that in order for this condition to be non-empty, we need to ask for the existence of a $(-2)$-root in $L$; otherwise all lattices of the form $L(n)$ with $n\ge 2$ would be Borcherds lattices, since they do not contain any $(-2)$-root.
Therefore it is not really restrictive to ask that $L$ contains infinitely many simple $(-2)$-roots, since if there were a simple $(-2)$-root but only finitely many, the symmetry group of $L$ would be finite by Proposition \ref{prop:finitely.many.roots}, and Nikulin \cite{nikulin.finite.aut.3, nikulin.finite.aut.greater.5} and Vinberg \cite{vinberg.finite.aut.4} already classified hyperbolic lattices with finite symmetry group. In particular, the symmetry group of Borcherds lattices is always infinite.

Already Borcherds in \cite{borcherds.leech} noticed that the unimodular lattice $\mathrm{II}_{1,25}=U\oplus E_8^3$ is a Borcherds lattice. Indeed there exists an isometry $\mathrm{II}_{1,25}\cong U\oplus \Lambda$, where $\Lambda$ is the \emph{Leech lattice}, i.e. the unique negative definite unimodular lattice of rank $24$ with no $(-2)$-root. It can be shown that the cusp $\vect{e}\in \mathrm{II}_{1,25}$ such that $\vect{e}^\perp/\langle \vect{e}\rangle \cong \Lambda$, has inner product $1$ with every simple $(-2)$-root of $\mathrm{II}_{1,25}$ (for instance, this can be obtained from Proposition \ref{prop:covering.radius}, recalling that the Leech lattice has covering radius $\sqrt{2}$ \cite[{\S}23,~Theorem~1]{conway.sloane.lattices}). This observation motivates the following definition:

\begin{definition}
A negative definite lattice $W$ is a \emph{Leech type lattice} if $U\oplus W$ is a Borcherds lattice.
\end{definition}

Before passing to the classification of Borcherds lattices, we want to present several equivalent characterizations of Borcherds lattices, tying together many important concepts.
Before stating the result, we need to recall the definition of the \emph{exceptional lattice}, as introduced by Nikulin \cite[{\S}4]{nikulin.elliptic}. If $L$ is a hyperbolic lattice, denote by $E(L)$ the subset of $L$ containing the vectors whose stabilizer in $\Aut(\mathcal{D}_L)$ has finite index. Clearly $E(L)$ is a sublattice of $L$, and it is called the \emph{exceptional lattice} of $L$.

The equivalence of conditions (b--e) in the following theorem can essentially be traced back to Nikulin \cite[Theorem~9.1.1]{nikulin.factor.groups} and \cite[{\S}4]{nikulin.elliptic}. It was picked up by Cantat \cite{cantat.dynamics.K3.1, cantat.dynamics.K3.2}, and it was explicitly stated by Oguiso in \cite[Theorems~1.6~and~2.1]{oguiso.entropy}. We propose an alternative proof that uses hyperbolic geometry.

\begin{theorem} \label{thm:equivalence.characterizations.Borcherds}
Let $L$ be a hyperbolic lattice with infinite symmetry group $\Aut(\mathcal{D}_L)$. The following are equivalent:
\begin{enumerate}[(a)]
    \item $L$ is a Borcherds lattice;
    \item $L$ has zero entropy;
    \item There is a unique cusp $\vect{e}$ of $\mathcal{D}_L$ with infinite stabilizer $\Stabe$;
    \item There is a cusp $\vect{e}$ of $\mathcal{D}_L$ such that $\Aut(\mathcal{D}_L)=\Stabe$;
    \item There is a cusp $\vect{e}$ of $\mathcal{D}_L$ such that $\Stabe$ has finite index in $\Aut(\mathcal{D}_L)$;
    \item The exceptional lattice $E(L)$ is parabolic, i.e. it is negative semidefinite with a $1$-dimensional kernel.
\end{enumerate}
\end{theorem}
\begin{proof}
    (a) $\Rightarrow$ (d): By assumption there is a cusp $\vect{e}$ of $\mathcal{D}_L$ such that $\vect{e}.\vect{r} \le N$ for any simple $(-2)$-root $\vect{r}\in L$. Let $g\in \Aut(\mathcal{D}_L)$ and assume by contradiction that $g(\vect{e})\ne \vect{e}$. Since $g(\vect{e})$ is fundamental, $\vect{h}=\vect{e}+g(\vect{e})$ has positive square.
    We have that $g(\vect{e}).\vect{r}=\vect{e}.g^{-1}(\vect{r}) \le N$ for any simple $(-2)$-root $r$, since $g^{-1}(\vect{r})$ is a simple $(-2)$-root as well. Therefore $\vect{h}.\vect{r}\le 2N$ for any simple $(-2)$-root $\vect{r}\in L$, and hence there are only finitely many simple $(-2)$-roots in $L$. This contradicts the definition of a Borcherds lattice.
    
    (d) $\Rightarrow$ (e): Obvious.
    
    (e) $\Rightarrow$ (f): The cusp $\vect{e}\in L$ belongs to the exceptional lattice, hence $E(L)$ is either hyperbolic or parabolic. Assume by contradiction that it is hyperbolic. Then $E(L)$ contains a vector $\vect{v}$ of positive norm, and by definition $\vect{v}$ is fixed by a subgroup $G$ of $\Aut(\mathcal{D}_L)$ of finite index. Since $\vect{v}$ has positive norm, $G$ is necessarily finite, and therefore $\Aut(\mathcal{D}_L)$ is finite as well, contradicting the initial assumption.
    
    (f) $\Rightarrow$ (b): By assumption there exists a cusp $\vect{e}\in L$ in the exceptional lattice, that is, $\Aut(\mathcal{D}_L)$ coincides with $\Stabe$ up to a finite group. For any $g\in \Aut(\mathcal{D}_L)$, there is an $n\ge 1$ such that $g^n$ preserves the element $\vect{e}$, and therefore $g^n$ has zero entropy. It follows that $g$ has zero entropy as well, since otherwise $g$ would have an eigenvalue of norm $>1$, and the same would hold for $g^n$.

    (b) $\Rightarrow$ (c): We claim first that there exists a cusp of $\mathcal{D}_L$ with infinite stabilizer. If $f\in \Aut(\mathcal{D}_L)$ is any element of infinite order, by assumption $f$ has zero entropy (or equivalently it is parabolic), and therefore $f$ fixes some cusp $\vect{e}\in \mathcal{D}_L$. We deduce that the stabilizer $\Stabe$ is infinite.
    Assume by contradiction that there is a second cusp $\vect{e'}\in \mathcal{D}_L$ with infinite stabilizer, and choose any $g\in \Stabe$, $g'\in \Stabee$ of infinite order. The subgroup $\Gamma=\langle g,g'\rangle$ of $\Aut(\mathcal{D}_L)$ contains only elliptic and parabolic isometries by assumption, and therefore by \cite[Theorem~12.2.3]{ratcliffe} all the isometries of $\Gamma$ fix the same cusp of $L$. Since $g$ and $g'$ fix a unique cusp of $\mathcal{D}_L$, we deduce that $\vect{e}=\vect{e'}$, a contradiction.
    
    (c) $\Rightarrow$ (a): We show first that $L$ contains a $(-2)$-root. Let $\vect{e}\in \mathcal{D}_L$ be the cusp with infinite stabilizer $\Stabe$, and take any $\vect{f}\in L$ with $\vect{e}.\vect{f}=n=\vect{e}.L$. If $\vect{f}^2=2k$, consider the isotropic vector $\vect{v_0}=-k\vect{e}+n\vect{f}$, and denote by $\vect{v}\in L$ the unique primitive isotropic vector in $\langle \vect{v_0} \rangle \otimes \QQ$ with $\vect{v}.\vect{e}>0$. Notice that $\vect{v}$ is positive, and by assumption it is either non-fundamental or it has finite stabilizer. In the first case there is a positive $(-2)$-root $\vect{r}$ such that $\vect{v}.\vect{r}<0$, while in the second there is a positive $(-2)$-root $\vect{r}$ such that $\vect{v}.\vect{r}=0$ by Proposition \ref{prop:generalization.Shioda.Tate}. In both cases, $L$ contains a $(-2)$-root. Since by assumption the symmetry group of $L$ is infinite, $L$ contains infinitely many simple $(-2)$-roots by Proposition \ref{prop:finitely.many.roots}.
        
    It remains to prove that the vector $\vect{e}\in L$ has bounded inner product with all the simple $(-2)$-roots of $L$. 
    Let $\vect{r_1},\ldots,\vect{r_m}$ be a set of representatives of the $\Aut(\mathcal{D}_L)$-orbits of simple $(-2)$-roots of $L$. Denote $N\coloneqq \max\{\vect{e}.\vect{r_i} : 1\le i \le m\}$.
    Since $\vect{e}$ is the only cusp of $\mathcal{D}_L$ with infinite stabilizer, clearly $\vect{e}$ is fixed by the whole symmetry group $\Aut(\mathcal{D}_L)$. Now, if $\vect{r}$ is any simple $(-2)$-root in $L$, by construction there exists an isometry $g\in \Aut(\mathcal{D}_L)$ and an $1\le i \le m$ such that $\vect{r}=g(\vect{r_i})$. But then
    $$\vect{e}.\vect{r}=g(\vect{e}).\vect{r}=\vect{e}.g^{-1}(\vect{r})=\vect{e}.\vect{r_i} \le N,$$
    as desired.
\end{proof}

\begin{remark}   
 A Borcherds lattice $L$ admits only one cusp with bounded inner product with all the simple $(-2)$-roots of $L$. Indeed, if there were two distinct ones, say $\vect{e}$ and $\vect{e'}$, then their sum $\vect{e}+\vect{e'}$ would be a fundamental vector of positive square with bounded inner product with all the simple $(-2)$-roots of $L$. This would imply the existence of only finitely many simple $(-2)$-roots, contradicting the definition of Borcherds lattices.

 Let $\vect{e}\in L$ be the cusp with bounded inner product will all the simple $(-2)$-roots of $L$. The vector $\vect{e}$ is also the unique cusp of $\mathcal{D}_L$ with infinite stabilizer, since it follows from the previous point that the symmetry group $\Aut(\mathcal{D}_L)$ fixes $\vect{e}$.
\end{remark}

A version of the following result was proved by Nikulin (see \cite[Theorem~9.1.1~and~its~preceding~discussion]{nikulin.factor.groups}). We show that in fact a slightly stronger version of the theorem holds:

\begin{theorem}\label{thm:virtually.abelian} 
Every Borcherds lattice $L$ has a virtually abelian symmetry group $\Aut(\mathcal{D}_L)$. Moreover every hyperbolic lattice $L$ with an infinite, virtually solvable symmetry group is a Borcherds lattice, assumed that $\rk(L)\ge 5$ or $\mathcal{D}_L$ admits a cusp with infinite stabilizer. 
\end{theorem}
\begin{proof}
The first statement follows from Theorem \ref{thm:equivalence.characterizations.Borcherds} and Proposition \ref{prop:generalization.Shioda.Tate}, since the symmetry group $\Aut(\mathcal{D}_L)$ coincides with $\Stabe$ for some cusp $\vect{e}\in \mathcal{D}_L$.

Before proving the converse implication, we recall that for a hyperbolic lattice $L$, $\Aut(\mathcal{D}_L)$ is virtually solvable if and only if it is virtually abelian (cf. \cite[Theorems~5.5.9,~5.5.10]{ratcliffe}).

We can now prove the second implication in the statement. Assume first that $\mathcal{D}_L$ admits a cusp with infinite stabilizer, that we denote $\vect{e}$.
Let $g\in \Aut(\mathcal{D}_L)$ be any element, and choose $g'\in \Stabe$ of infinite order. Since $\Aut(\mathcal{D}_L)$ is virtually abelian, there is an $n\ge 1$, independent of $g$ and $g'$, such that $g^n$ and $g'^n$ commute. Then $g'^n\circ g^n(\vect{e})=g^n \circ g'^n(\vect{e})=g^n(\vect{e})$, so $g'^n$ preserves the isotropic vector $g^n(\vect{e})$. If by contradiction $g^n(\vect{e})\ne \vect{e}$, then $g^n(\vect{e}).\vect{e}>0$ and $g'^n$ preserves the element $\vect{e}+g^n(\vect{e})$ of positive square. This however contradicts the fact that $g^n$ has infinite order.
We deduce that $g^n(\vect{e})=\vect{e}$, and in particular that $\Aut(\mathcal{D}_L)$ coincides up to a finite group with $\Stabe$. Hence $L$ is a Borcherds lattice by Theorem \ref{thm:equivalence.characterizations.Borcherds}.

This is sufficient to prove the statement if $\rk(L)\ge 6$, since every hyperbolic lattice of rank $\ge 6$ with infinite symmetry group admits an isotropic vector with infinite stabilizer \cite[Theorem~6.4.1]{nikulin.factor.groups}. Moreover the only hyperbolic lattices of rank $5$ with infinite symmetry group, such that the stabilizer of every isotropic vector is finite, are those of the form $\langle 2^m\rangle \oplus D_4$, with $m\ge 5$, and $\langle 2\cdot 3^{2m-1}\rangle \oplus A_2^2$, with $m\ge 2$, and a direct calculation shows that their symmetry groups are not virtually abelian (see the proof of \cite[Theorem~9.1.1]{nikulin.factor.groups} for more details).
\end{proof}

\begin{remark} \label{rk:useful.remarks.borcherds}
\begin{itemize}
    \item Every Borcherds lattice has rank $\ge 3$. Indeed, every hyperbolic isotropic lattice of rank $2$ %has rank at least $2$, and if the rank is $2$ it 
    has finite symmetry group (see for instance \cite[Corollary~3.4]{galluzzi.lombardo.peters}) and in particular finitely many simple $(-2)$-roots.
   
    \item It is not true that all hyperbolic lattices $L$ of rank $\le 4$ with an infinite, but virtually solvable symmetry group $\Aut(\mathcal{D}_L)$ are Borcherds lattices. For instance every hyperbolic lattice of rank $2$ has a virtually abelian symmetry group \cite[Corollary~3.4]{galluzzi.lombardo.peters}, but most of them are not even isotropic.
    Moreover the lattice $L=U(20)\oplus \langle -2\rangle$ is isotropic and it has positive entropy, but its symmetry group is virtually abelian. The same happens for the lattice $L=U(6)\oplus A_1^2$ in rank $4$ (we compute their symmetry group via Borcherds' method explained in Section \ref{sec:borcherds.method}, and we check that it is virtually abelian by using the algorithm described in \cite{detinko.flannery.obrien.tits} and its implementation in Magma \cite{magma}).
    Nevertheless, the following proposition shows that if $L$ has positive entropy and a virtually solvable symmetry group, the rank of $\Aut(\mathcal{D}_L)$ must be $1$.
\end{itemize}
\end{remark}

\begin{proposition} \label{prop:virtually.abelian.positive.entropy} \cite[Theorems~5.5.9,~5.5.10]{ratcliffe}
Let $L$ be a hyperbolic lattice with an infinite, but virtually solvable symmetry group. Then either $L$ is a Borcherds lattice or $\Aut(\mathcal{D}_L)$ is virtually cyclic, i.e. it contains a subgroup of finite index isomorphic to $\ZZ$. 
\end{proposition}

\section{The classification} \label{sec:classification}

The goal of this section is to classify Borcherds and Leech type lattices. 

\subsection{A structure lemma for isotropic hyperbolic lattices}

In the following let $L$ be an \emph{isotropic} hyperbolic lattice, i.e. a hyperbolic lattice containing an isotropic vector.

\begin{lemma} \label{lemma:structure.lemma}
For any isotropic vector $\vect{e}\in L$, there exists a basis $\mathcal{B}=\{\vect{e},\vect{f},\vect{w_1},\ldots,\vect{w_r}\}$ of $L$ such that the corresponding Gram matrix is of the form
\begin{equation} \label{eq:matrix}
\left(\begin{array}{cc|ccc}
0 &n &0 &\ldots &0\\
n &2k & &\underline{\ell}^T &\\
\hline
0 & & & &\\
\vdots &\underline{\ell} & &W &\\
0 & & & &
\end{array}\right),
\end{equation}
where $-n\le k < n$, $0\le \ell_i < n$ for each entry $\ell_i$ of $\underline{\ell}$, and $W$ is the Gram matrix of the negative definite lattice $\vect{e}^\perp/\langle \vect{e}\rangle$.
\end{lemma}
\begin{proof}
Let $\vect{e}\in L$ be a primitive isotropic vector, with index $n=\vect{e}.L$. There exists a primitive vector $\vect{f}\in L$ with $\vect{e}.\vect{f}=n$, and up to changing $\vect{f}$ with $\vect{f}+\alpha \vect{e}$ for some $\alpha\in \ZZ$, we may assume that $\vect{f}^2=2k\in [-n,n)$.

The sublattice $H=\langle \vect{e},\vect{f}\rangle$ is primitive in $L$: if it were not, its saturation $H_{sat}$ in $L$ would contain a vector $\vect{f'}=\frac{1}{c}(a \vect{e}+b\vect{f})$ with $0<b<c$, and thus $\vect{e}.\vect{f'}<n$, contradicting the minimality of $n$.

Hence we may extend $\{\vect{e},\vect{f}\}$ to a basis $\{\vect{e},\vect{f},\vect{v_1},\ldots,\vect{v_r}\}$ of $L$. Since by assumption the index of $\vect{e}$ is $n$, the products $\vect{e}.\vect{v_i}$ are multiples of $n$ for every $1\le i\le r$. In particular we can substitute $\vect{v_i}$ with $\vect{w_i}\coloneqq \vect{v_i}- \frac{\vect{e}.\vect{v_i}}{n}\vect{f}$ and obtain a Gram matrix for $L$ as in (\ref{eq:matrix}). Up to substituting $\vect{w_i}$ with $\vect{w_i}+\alpha \vect{e}$ for some $\alpha\in \ZZ$, we may assume that $0\le \ell_i < n$ for each entry $\ell_i$ of $\underline{\ell}$.

In order to conclude the proof, notice that the sublattice $\vect{e}^\perp$ of $L$ admits the basis $\{\vect{e},\vect{w_1},\ldots,\vect{w_r}\}$. Hence $\vect{w_1},\ldots,\vect{w_r}$ descend to a basis of $\vect{e}^\perp /\langle \vect{e}\rangle$, showing that in fact $W$ is the Gram matrix of the lattice $ \vect{e}^\perp / \langle \vect{e}\rangle$.
\end{proof}

The choice of a basis for $L$ as above is convenient to compute the inner product of two given vectors. The following computation will be useful in the paper:

\begin{lemma} \label{lemma:computation}
Let $L$ be a hyperbolic lattice with Gram matrix as in (\ref{eq:matrix}), with basis $\{\vect{e},\vect{f},\vect{w_1},\ldots,\vect{w_r}\}$.
Let $\vect{v}=\alpha \vect{e} + \beta \vect{f} + \vect{\gamma}\in L$ be a vector of square $2N$ and $\vect{w}=x\vect{e}+y\vect{f}+\vect{z}\in L$ a vector of square $2M$, where $\alpha,\beta,x,y\in \ZZ$, $\beta,y\ne 0$, and $\vect{\gamma},\vect{z}\in W$. Then
$$\vect{v}.\vect{w} = \frac{1}{\beta y}\left( -\frac{1}{2}(y\vect{\gamma}-\beta\vect{z})^2 + Ny^2 + M\beta^2 \right).$$
\end{lemma}
\begin{proof}
By assumption we have
$$\begin{cases}
2N=\vect{v}^2 = 2n\alpha\beta + 2k\beta^2 + 2\beta \underline{\ell}^T . \underline{\gamma} + \underline{\gamma}^T.W.\underline{\gamma}\\
2M = \vect{w}^2 = 2nxy +2ky^2+2y \underline{\ell}^T. \underline{z} + \underline{z}^T.W.\underline{z}
\end{cases}$$
where $\underline{\gamma}$ (resp. $\underline{z}$) is the column vector of coefficients of $\vect{\gamma}$ (resp. $\vect{z}$) in $W$ with respect to the chosen basis.
A straightforward computation shows
$$\vect{v}.\vect{w} = n\alpha y + n\beta x +2k \beta y +\beta \underline{\ell}^T.\underline{z} + y \underline{\ell}^T.\underline{\gamma}+\underline{\gamma}^T.W.\underline{z}=\frac{1}{\beta y}\left( -\frac{1}{2}(y\vect{\gamma}-\beta\vect{z})^2 + Ny^2 + M\beta^2 \right).$$

%We deduce
%\begin{equation*}
%\begin{split}
%     \vect{v}.\vect{w} &= n\alpha y + n\beta x +2k \beta y +\beta \underline{\ell}^T.\underline{z} + y \underline{\ell}^T.\underline{\gamma}+\underline{\gamma}^T.W.\underline{z} =\\
%     &=\left( -k\beta -\underline{\ell}^T.\underline{\gamma}-\frac{1}{2\beta}\underline{\gamma}^T.W.\underline{\gamma} + \frac{N}{\beta} \right)y + \left( -ky-\underline{\ell}^T.\underline{z}-\frac{1}{2y}\underline{z}^T.W.\underline{z}+\frac{M}{y} \right)\beta +\\
%     &\quad +2k \beta y +\beta \underline{\ell}^T.\underline{z} + y \underline{\ell}^T.\underline{\gamma}+\underline{\gamma}^T.W.\underline{z}=\\
%     &= \frac{1}{\beta y} \left( -\frac{1}{2}(y \underline{\gamma})^T.W.(y\underline{\gamma}) -\frac{1}{2}(\beta \underline{z})^T.W.(\beta\underline{z}) + (y\underline{\gamma})^T.W.(\beta\underline{z}) + Ny^2 + M\beta^2 \right) =\\
%     &=\frac{1}{\beta y} \left( -\frac{1}{2}(y \underline{\gamma}-\beta \underline{z})^T.W.(y\underline{\gamma}-\beta\underline{z}) + Ny^2 + M\beta^2 \right)=\\
%     &=\frac{1}{\beta y}\left( -\frac{1}{2}(y\vect{\gamma}-\beta\vect{z})^2 + Ny^2 + M\beta^2 \right),
% \end{split}
% \end{equation*}
%as claimed.
\end{proof}

\subsection{The classification of Leech type lattices}

Recall that a negative definite lattice $W$ is a Leech type lattice if $U\oplus W$ is a Borcherds lattice, or, equivalently by Theorem \ref{thm:equivalence.characterizations.Borcherds}, if the symmetry group of $U\oplus W$ is infinite and of zero entropy. Since any lattice $W'$ in the genus of $W$ gives rise to an isometric hyperbolic lattice $L=U\oplus W \cong U\oplus W'$, we will classify Leech type lattices by including only one representative for each genus. The final result is as follows:

\begin{theorem} \label{thm:leech.classification}
There are $172$ distinct (genera of) Leech type lattices, and the list can be found in the ancillary files.
\end{theorem}

The classification can be obtained by combining the partial classifications in Sections \ref{sec:root.overlattices.Leech} and \ref{sec:Leech.lattices.unique.genus}.
First, we state the following interesting consequences of Theorem \ref{thm:leech.classification}:

\begin{corollary}
\begin{enumerate}[(a)]
    \item Every Leech type lattice embeds primitively in some unimodular negative definite lattice of rank $24$.
    \item The Leech lattice is the only lattice of Leech type that is not unique in its genus and that contains no $(-2)$-root.
\end{enumerate}
\end{corollary}

We start with a sufficient and a necessary condition for a negative definite lattice to be of Leech type. Recall that the \emph{covering radius} of a positive definite lattice $P$ is the smallest $r>0$ with the property that, for any $\vect{q}_\RR\in P\otimes \RR$, there is $\vect{p}\in P$ such that $\sqrt{(\vect{q}_\RR-\vect{p})^2} \le r$.

Conway's \cite[Chapter 27, Theorem 1]{conway.sloane.lattices} proof that the Leech lattice is of Leech type leads to the following slight generalization.
\begin{proposition} \label{prop:covering.radius}
Let $W$ be a negative definite lattice that is not a root overlattice, and such that $W(-1)$ has covering radius $\le \sqrt{2}$. Then $W$ is a Leech type lattice.
\end{proposition}
\begin{proof}
We need to show that the hyperbolic lattice $L\coloneqq U\oplus W$ is a Borcherds lattice. Surely $L$ contains a $(-2)$-root, and $\Aut(\mathcal{D}_L)$ is infinite by Proposition \ref{prop:generalization.Shioda.Tate}. Let $\{\vect{e},\vect{f}\}$ be the basis of $U$ such that $\vect{e}^2=0$, $\vect{f}^2=-2$ and $\vect{e}.\vect{f}=1$. Without loss of generality, we may assume that $\vect{e}$ is a fundamental vector of $L$. We claim that for every simple $(-2)$-root $\vect{r}$ of $L$, the inner product $\vect{e}.\vect{r}$ is bounded from above by $1$. Equivalently, we claim that every positive $(-2)$-root $\vect{r}\in L$ with $\vect{e}.\vect{r} \ge 2$ is not simple.

Let $\vect{r}=x \vect{e} + y \vect{f} + \vect{z} \in L$ be a positive $(-2)$-root with $y=\vect{e}.\vect{r}\ge 2$. In order to show that $\vect{r}$ is not simple, we are going to exhibit a positive $(-2)$-root $\vect{r'}\in L$ with $\vect{e}.\vect{r'}=1$ such that $\vect{r}.\vect{r'} < 0$. This is sufficient because it implies that $\vect{r}-\vect{r'}$ is positive (since $(\vect{r}-\vect{r'})^2 \ge -2$ and $\vect{e}.(\vect{r}-\vect{r'})>0$), contradicting the simplicity of $\vect{r}$. 

Consider the vector $\frac{\vect{z}}{y}\in W\otimes \RR$. Since $W(-1)$ has covering radius $\le \sqrt{2}$, there exists a vector $\vect{z'}\in W$ such that $-\left(\frac{\vect{z}}{y}-\vect{z'}\right)^2 \le 2$. Let $x'\coloneqq -\frac{1}{2}\vect{z'}^2$. It is straightforward to check that $\vect{r'}\coloneqq x' \vect{e} + \vect{f} + \vect{z'}\in L$ is a positive $(-2)$-root with $\vect{e}.\vect{r'}=1$. We claim that $\vect{r}.\vect{r'}<0$. Indeed by Lemma \ref{lemma:computation} we have
\begin{equation*}
\begin{split}
    \vect{r}.\vect{r'} &=\frac{1}{y}\left( -\frac{1}{2}(y\vect{z'}-\vect{z})^2-y^2-1\right) 
    =y \left( -\frac{1}{2} \left(\frac{\vect{z}}{y}-\vect{z'}\right)^2-1-\frac{1}{y^2} \right) <\\
    &<\frac{y}{2} \left( - \left(\frac{\vect{z}}{y}-\vect{z'}\right)^2-2\right) \le 0,
\end{split}
\end{equation*}
as desired.
\end{proof}

\begin{proposition} \label{prop:genus.Leech.type.lattice}
Let $W$ be a Leech type lattice. The genus of $W$ contains precisely one lattice that is not a root overlattice.
\end{proposition}
\begin{proof}
If the genus of $W$ only contains root overlattices, then by \cite[Theorems~3.1.1~and~4.1.1]{nikulin.factor.groups} the lattice $U\oplus W$ has a finite symmetry group, so we may assume that the genus of $W$ contains at least one lattice that is not a root overlattice.
Assume by contradiction that the genus of $W$ contains two non-isometric lattices $W_1$ and $W_2$ that are not root overlattices. Since $L=U\oplus W$ is isometric to both $U\oplus W_1$ and $U\oplus W_2$, there are two cusps $\vect{e_1}, \vect{e_2}$ of $\mathcal{D}_L$ such that $\vect{e_i}^\perp /\langle \vect{e_i} \rangle \cong W_i$ for $i=1,2$. Since $W_1$ and $W_2$ are not root overlattices it follows that both $\vect{e_1}$ and $\vect{e_2}$ have infinite stabilizers by Proposition \ref{prop:generalization.Shioda.Tate}, contradicting Theorem \ref{thm:equivalence.characterizations.Borcherds}.
\end{proof}

As a consequence we only have two possibilities for a Leech type lattice $W$: either it is unique in its genus (and it is not a root overlattice), or its genus contains a root overlattice. We are going to treat the two cases separately.

\subsection{Root overlattices of Leech type} \label{sec:root.overlattices.Leech}

The first step towards Theorem \ref{thm:leech.classification} consists of a concrete computation. Since there are only finitely many root overlattices in each rank, we can list all root overlattices of Leech type of rank $\le 24$.

In order to decide whether a root overlattice $R$ is a Leech lattice, we proceed as follows. First, we check that the genus of $R$ contains at least one lattice that is not a root overlattice. In fact this is the case if and only if the lattice $U\oplus R$ has an infinite symmetry group by \cite[Theorems~3.1.1~and~4.1.1]{nikulin.factor.groups}, and Nikulin and Vinberg have classified the hyperbolic lattices with finite symmetry group. For the convenience of the reader, we list in Table \ref{tab:2.reflective.root.overlattices} the root overlattices $R$ such that $U\oplus R$ has finite symmetry group (they are called \emph{$2$-reflective lattices}).

Secondly, we compute the covering radius of $W(-1)$. By Proposition \ref{prop:covering.radius}, if the covering radius is at most $\sqrt{2}$, then $W$ is of Leech type. The (squares of the) covering radii of some Leech type lattices can be found in the third column of Table \ref{tab:root.overlattices.Leech.type}.

As a third step, we check whether there are two non-isometric lattices in the genus of $R$ that are not root overlattices. This can be done in a computationally fast way by looking at $2$-, $3$- and $5$-neighbors of $R$. If we are able to find two such neighbors, then $R$ is not of Leech type by Proposition \ref{prop:genus.Leech.type.lattice}.

As a fourth and final step, we use Borcherds' method to decide whether $U\oplus R$ is a Borcherds lattice, or equivalently if $R$ is of Leech type.

We are now in a position to prove that the list in Table \ref{tab:root.overlattices.Leech.type} is in fact complete.

\begin{proposition} \label{prop:no.leech.rank.25}
There are no root overlattices of Leech type of rank $\ge 25$.
\end{proposition}
\begin{proof}
Let $W$ be a Leech type lattice of rank $r\ge 25$ and length $\ell=\ell(A_W)$. For any overlattice $W'$ of $W$, we have that $U\oplus W'$ has zero entropy by Proposition \ref{prop:entropy.overlattice}, since $U\oplus W$ has zero entropy by Theorem \ref{thm:equivalence.characterizations.Borcherds}. Moreover $U\oplus W'$ has infinite symmetry group, since by Nikulin's classification every hyperbolic lattice of rank $\ge 20$ has infinite symmetry group \cite{nikulin.finite.aut.greater.5}. Therefore $U\oplus W'$ is a Borcherds lattice as well, and up to substituting $W$ with one of its maximal overlattices, we may assume that $W$ has no non-trivial overlattices.
This implies that $\ell \le 3$ by \cite[Lemma~3.5.3]{kirschmer}. Since $r-\ell > 16$, then, up to substituting $W$ with another lattice in its genus, there exists a primitive embedding $E_8^2 \hookrightarrow W$. This follows essentially by \cite[Corollary~1.13.5]{nikulin.integral.symmetric.bilinear}: in fact there exists an embedding $U\oplus E_8^2 \hookrightarrow U\oplus W$, hence $W$ is in the same genus as $E_8^2\oplus R$ for a certain negative definite lattice $R$. It follows that $R$ is of rank $\ge 9$ and length $\ell\le 3$.

We claim that the genus of $R$ contains only root overlattices: indeed, if there is a non-root overlattice $M$ in the genus of $R$, then $W_1=E_8^2\oplus M$ and $W_2=D_{16}^+\oplus M$ are in the genus of $W$, they are not root overlattices and they are not isometric (for instance, the root parts have different discriminants). Here $D_{16}^+$ denotes the negative definite, unimodular lattice of rank $16$ with root part isometric to $D_{16}$. We deduce that $U\oplus R$ has finite symmetry group by \cite[Theorems~3.1.1~and~4.1.1]{nikulin.factor.groups}, and therefore $R$ is one of the root overlattices of rank $\ge 9$ in Table \ref{tab:2.reflective.root.overlattices}. We check that in all these cases each lattice $W=E_8^2\oplus R$  admits two distinct non-root overlattices in its genus.
Therefore is not of Leech type.
\end{proof}

\subsection{Leech type lattices unique in their genus} \label{sec:Leech.lattices.unique.genus}

In this section we assume instead that $W$ is unique in its genus. The \emph{scale} of a lattice is the greatest common divisor of the entries of its Gram matrix (with respect to any basis). By \cite{voight, lorch.kirschmer.single.class} we have a complete and finite list of (possibly odd) negative definite lattices unique in their genus of scale $1$. 
Since a lattice $W$ is unique in its genus if and only if all (or one of) its multiples is unique in its genus, we have an explicit list of negative definite lattices unique in their genus.

The following proposition, which relies on \cite[Theorem 4.6]{mezzedimi.entropy}, is the key to show that only a finite number of multiples of a given lattice can be of Leech type.

\begin{proposition} \label{prop:bound.multiples}
Let $W$ be a negative definite lattice of rank $r\ge 2$ and unique in its genus. The only multiples of $W$ that can be of Leech type are the $W(m)$ for $m \le N$, where $N>0$ is an explicit constant.
More precisely, $N$ can be computed as follows: Fix any primitive sublattice $T$ of $W$ of corank $1$, and set $N\coloneqq \max\{a,b\}$, where:
\begin{itemize}
    \item The constant $a>0$ is such that $T(m)$ is not of Leech type for any $m\ge a$;
    \item $b\coloneqq \lfloor \frac{2\disc(T)}{\disc(W)} \rfloor$.
\end{itemize}
\end{proposition}
\begin{proof}
Fix any $m>N$. Since $a\ge 1$ and $m>N\ge a$, necessarily $m$ is at least $2$.
We need to show that $U\oplus W(m)$ is not a Borcherds lattice, or equivalently that it has positive entropy by Theorem \ref{thm:equivalence.characterizations.Borcherds}.
Notice that the lattice $U\oplus W(m)$ satisfies the assumptions in \cite[Theorem 4.6]{mezzedimi.entropy}. Indeed, since $m\ge 2$, the lattice $W(m)$ has no $(-2)$-roots.
Moreover consider the primitive sublattice $T(m)$ of $W(m)$. By assumption the hyperbolic lattice $U\oplus T(m)$ has positive entropy, and moreover
\begin{equation*}
\begin{split}
    \disc(W(m)) &=m^r \disc(W) \ge (b+1)m^{r-1}\disc(W) > \left( \frac{2\disc(T)}{\disc(W)}\right)   m^{r-1}\disc(W) = \\
    &=2m^{r-1}\disc(T)=2\disc(T(m)),
\end{split}
\end{equation*}
so we conclude that $U\oplus W(m)$ has positive entropy by \cite[Theorem 4.6]{mezzedimi.entropy}.
\end{proof}

Proposition \ref{prop:bound.multiples} suggests a recursive approach, since for lattices of rank $r\ge 2$ the constant $N$ can be explicitly computed only if we already have a complete list of Leech type lattices of rank $r-1$. For this reason we need to deal with the case of rank $1$ first. The classification is as follows (see also \cite[Theorem~5.10]{mezzedimi.entropy}):

\begin{proposition}[{\cite[Theorem~3~and~the~subsequent~discussion]{nikulin.interesting}}]
The Leech type lattices of rank $1$ are those of the form $\langle -2k\rangle$ for $k\in \{2,3,4,5,7,9,13,25\}$.
\end{proposition}

\begin{remark}
As noted by X. Roulleau, the list of $k\ge 2$ for which $\langle -2k\rangle$ is of Leech type coincides with the list of $k\ge 2$ such that $k-1$ divides $24$. At the moment we do not have any explanation for this phenomenon.
\end{remark}

We now have all the necessary ingredients to complete the classification of Leech type lattices. By the classification in \cite{lorch.kirschmer.single.class}, negative definite lattices unique in their genus have rank $\le 10$. Therefore, for each $2\le r \le 10$, we recursively list all Leech type lattices that are unique in their genus as follows.
We take the (finite) list of negative definite lattices of rank $r$ and scale $1$ that are unique in their genus (if a lattice is odd, we just multiply it by $2$). Since we already have a complete list of Leech type lattices of rank $r-1$, we use Proposition \ref{prop:bound.multiples} to find, for each lattice $W$, a constant $N_W$ such that $W(m)$ is not of Leech type for any $m> N_W$. This produces a finite list of lattices, and we employ the same strategy as in Section \ref{sec:root.overlattices.Leech} in order to single out the Leech type lattices among these.
This concludes the proof of Theorem \ref{thm:leech.classification}.

\subsection{Independence of the Generalized Riemann Hypothesis} \label{sec:GRH}

As seen in Section \ref{sec:Leech.lattices.unique.genus}, our classification of Leech type lattices uses the classification of definite lattices unique in their genus, which in turn depends on the Generalized Riemann Hypothesis (GRH) (cf. \cite{voight}). More precisely, there could be an extra definite lattice of rank $2$ unique in its genus (but its discriminant must be very big).
We explain in this section how to avoid the classification of definite lattices unique in their genus in rank $2$, and make all our statements independent of the GRH.

In this section $W$ will be a negative definite lattice of rank $2$, and more precisely
$$
W=\begin{pmatrix}
-2k_1 &a\\
a &-2k_2
\end{pmatrix}
$$
with $k_1\ge k_2\ge a\ge 0$ (this can be achieved up to isometry of $L$). In order to find a way around the GRH, we need to prove that $W$ is not of Leech type if $\disc(W)$ is big enough. This is done in \cite[Theorem~6.1]{mezzedimi.entropy} in the case that $k_1\ge k_2\ge 2$. The main point is that $\disc(W)\ge 4k_2$, hence every fundamental isotropic vector on $U\oplus \langle -2k_1\rangle$ or $U\oplus \langle -2k_2\rangle$ extends to a fundamental isotropic vector on $L=U\oplus W$. In particular $L$ has positive entropy as soon as one of $k_1$ and $k_2$ does not belong to $\{2,3,4,5,7,9,13,25\}$.

It remains to consider the case $k_2=1$. This case was not treated in \cite{mezzedimi.entropy}, and the previous approach fails, since $\disc(W)=4k_1-a^2$ can be less than $4k_1$.
We fix the following notation: $\{\vect{e},\vect{f}\}$ is a basis of $U$ such that $\vect{e}^2=0$, $\vect{f}^2=-2$ and $\vect{e}.\vect{f}=1$, $\{\vect{w_1},\vect{w_2}\}$ is a basis of $W$ whose associated Gram matrix is $\begin{pmatrix}
-2k &a\\
a &-2
\end{pmatrix}$, and we consider the rank $1$ lattice $\langle -2k\rangle$ as the primitive sublattice of $W$ generated by $\vect{w_1}$. Without loss of generality, we may assume that $\vect{w_2}$ is a positive $(-2)$-root of $L$.

\begin{proposition}
If $k\notin \{2,3,4,5,7,9,13,25\}$, then $W$ is not a lattice of Leech type.
\end{proposition}
\begin{proof}
The idea is to construct a fundamental primitive isotropic vector $\vect{v}\in L_1 = U\oplus \langle -2k\rangle$ with infinite stabilizer, extend it to $L=U\oplus W$ and check that it remains fundamental with infinite stabilizer. We follow the construction in the proof of \cite[Proposition~5.7]{mezzedimi.entropy}.

Assume first that we can write $k=pq$ with $p<q$ and $p$ is the smallest prime number dividing $k$ (this can be achieved if $k$ is not a prime nor the square of a prime).
It is straightforward to check that $\vect{v}=(p+q)\vect{e}+p\vect{f}+\vect{w_1}\in L_1$ is primitive and isotropic. We claim that $\vect{v}$ is fundamental considered as a vector of $L$. Let $\vect{r}= x \vect{e}+y \vect{f}+z_1 \vect{w_1}+z_2 \vect{w_2}\in L$ be any positive $(-2)$-root. If $y=\vect{e}.\vect{r}=0$, then $\vect{r}$ is orthogonal to $\vect{e}$ and thus $\vect{r}=\vect{w_2}$. However $\vect{v}.\vect{w_2}= a \ge 0$, so we may assume that $y>0$.

It follows by Lemma \ref{lemma:computation} that $\vect{v}.\vect{r}=-\frac{1}{2}\vect{t}^2-p^2$ up to a positive constant, where $\vect{t}=(y -p z_1)\vect{w_1} -p z_2 \vect{w_2}$. It is straightforward to check that $-\vect{t}^2 \ge 2p^2$. Indeed, if $y-pz_1 \ne 0$, we use the fact that any vector in $W$ with nonzero first coordinate has norm $\le - 2k \le -p^2$. If instead $y-pz_1=0$, we just need to observe that $z_2\ne 0$ (since if $y=pz_1$ and $z_2=0$, then the equation $\vect{r}^2=-2$ reads $xy-y^2-kz_1^2=-1$, and $p$ divides the left-hand side, a contradiction).

Finally, we have to show that $\vect{v}$ has infinite stabilizer in $L$. Since $\vect{v}.\vect{e}=p$, we have that $\vect{v}.L$ is either $1$ or $p$. This means that we can extend $\vect{e}$ to a basis of $L$ whose associated Gram matrix is as in (\ref{eq:matrix}), with $n\in \{1,p\}$. In both cases $\disc(\vect{v}^\perp / \langle \vect{v}\rangle)=\frac{\disc(L)}{n^2}=\frac{4k-a^2}{n^2} \ge \frac{4p(p+1)-1}{n^2} \ge \frac{4p^2+1}{n^2}  > 4$ by assumption, and therefore $\vect{v}$ has infinite stabilizer by Proposition \ref{prop:generalization.Shioda.Tate} (since all root overlattices of rank $2$ have discriminant $\le 4$).

Assume instead that $k$ is either a prime or the square of a prime. By \cite[Lemma~5.8]{mezzedimi.entropy} we can find $q\ge 2$ such that $q^2<k$, $q\nmid k-1$ and $(p,q)=1$. A completely analogous argument shows that $\vect{v}=(q^2+k)\vect{e}+q^2\vect{f}+q\vect{w_1}\in L$ is a fundamental primitive isotropic vector with infinite stabilizer.
\end{proof}

The above discussion ensures that the only negative definite lattices of rank $2$ that can be of Leech type are those with Gram matrix $
W=\begin{pmatrix}
-2k_1 &a\\
a &-2k_2
\end{pmatrix}
$ and $k_1,k_2\in \{1,2,3,4,5,7,9,13,25\}$. In particular we can bypass the classification of definite lattices of rank $2$ unique in their genus, making our results independent of the GRH.

\subsection{The classification of Borcherds lattices}

In this section we tackle the main problem of the paper, namely the problem of classifying Borcherds lattices. In the previous sections we classified Leech type lattices, or equivalently Borcherds lattices that contain a copy of the hyperbolic plane $U$, and we will see now how to use that classification to obtain our main result:

\begin{theorem} \label{thm:classification.Borcherds.lattices}
There are $194$ Borcherds lattices up to isometry, and the list can be found in the ancillary file.
\end{theorem}

Let us state a few easy consequences of this explicit classification, which answer some questions raised by Borcherds in \cite{borcherds.leech}:

\begin{corollary}\hfill
\begin{enumerate}[(a)]
    \item Every Borcherds lattice embeds primitively into the unimodular lattice $\mathrm{II}_{1,25}$.
    \item The unimodular lattice $\mathrm{II}_{1,25}$ is the only Borcherds lattice of rank $\ge 19$. In particular every hyperbolic lattice of rank $\ge 20$ and not isometric to $\mathrm{II}_{1,25}$ has positive entropy.
    \item If $L$ is a hyperbolic lattice with a virtually solvable symmetry group, then $\Aut(\mathcal{D}_L)$ contains a subgroup of finite index isomorphic to $\ZZ^m$, with $m\le 24$.
\end{enumerate}
\end{corollary}
\begin{proof}
For the first point, notice that for every Borcherds lattice $L\ne \mathrm{II}_{1,25}$ in the classification it holds $\rk(L)+\ell(A_L) < 26$. Hence by \cite[Corollary~1.12.3]{nikulin.integral.symmetric.bilinear} every Borcherds lattice admits an embedding into a unimodular lattice of signature $(1,25)$, and $\mathrm{II}_{1,25}$ is the unique such lattice up to isometry.

The second point follows from a direct inspection of the list of Borcherds lattices and from the fact that every hyperbolic lattice of rank $\ge 20$ has an infinite symmetry group \cite{nikulin.finite.aut.greater.5}.

Finally the last point follows from the fact that every Borcherds lattice has rank $\le 26$, by combining Theorem \ref{thm:equivalence.characterizations.Borcherds} and Propositions \ref{prop:generalization.Shioda.Tate} and \ref{prop:virtually.abelian.positive.entropy}.
\end{proof}

In the following $L$ is a Borcherds lattice, or equivalently a hyperbolic lattice of zero entropy with infinite automorphism group by Theorem \ref{thm:equivalence.characterizations.Borcherds}.
We start with a structure result for Borcherds lattices.

\begin{proposition} \label{prop:Gram.matrix.Borcherds.lattices}
Let $L$ be a Borcherds lattice. There exists a basis $\mathcal{B}=\{\vect{e},\vect{f},\vect{w_1},\ldots,\vect{w_r}\}$ of $L$ such that its Gram matrix is 
\begin{equation*}
\left(\begin{array}{cc|ccc}
0 &n &0 &\ldots &0\\
n &2k & &\underline{\ell}^T &\\
\hline
0 & & & &\\
\vdots &\underline{\ell} & &W &\\
0 & & & &
\end{array}\right)
\end{equation*}
as in (\ref{eq:matrix}) and such that:
\begin{enumerate}[(a)]
    \item $k=-1$;
    \item $0\le \ell_i \le n-1$ for every entry $\ell_i$ of $\underline{\ell}$;
    \item $W$ is a Leech type lattice and not a root overlattice;
    \item $n$ divides the scale of $W$.
\end{enumerate}
\end{proposition}
\begin{proof}
By Theorem \ref{thm:equivalence.characterizations.Borcherds} there exists a cusp $\vect{e}\in L$ with infinite stabilizer $\Stabe$, and by Lemma \ref{lemma:structure.lemma} we can find a basis $\mathcal{B}=\{\vect{e},\vect{f},\vect{w_1},\ldots,\vect{w_r}\}$ of $L$ whose associated Gram matrix is as in (\ref{eq:matrix}). By Proposition \ref{prop:generalization.Shioda.Tate} we have that $W\cong \vect{e}^\perp / \langle \vect{e} \rangle$ is not a root overlattice, since the stabilizer $\Stabe$ is infinite. We are going to show that the Gram matrix of $L$ satisfies the four conditions in the statement.

    \begin{enumerate}[(a)]
        \item Assume by contradiction that $k\ne -1$, and consider the isotropic vector $\vect{v_0}=-k\vect{e} + n\vect{f} \in L$. If $d=\gcd(k,n)$, the vector $\vect{v}=\vect{v_0}/d$ is also primitive. First we notice that $\vect{v}^\perp/\langle\vect{v}\rangle$ is not a root overlattice. Indeed by Lemma \ref{lemma:computation}, a $(-2)$-root $\vect{r}=x\vect{e}+y\vect{f}+\vect{z}\in L$ is orthogonal to $\vect{v}$ (or equivalently to $\vect{v_0}$) if and only if 
    $$-\frac{1}{2}(-n\vect{z})^2 - n^2 = 0, $$
    that is, if and only if $\vect{z}$ is a $(-2)$-root in $W$. Hence there is a homomorphism $(\vect{v}^\perp)_{root} \rightarrow W_{root}$ sending $\vect{r}$ to its component $\vect{z}\in W$. Let $\vect{r}=x\vect{e}+y\vect{f}+\vect{z}$ and $\vect{r'}=x'\vect{e}+y'\vect{f}+\vect{z}$ be $(-2)$-roots orthogonal to $\vect{v}$ with the same component with respect to $W$. The equations $\vect{v_0}.\vect{r}=\vect{v_0}.\vect{r'}=0$ read 
    $$n^2x+nky+n\underline{\ell}^T. \underline{z}=n^2x'+nky'+n\underline{\ell}^T. \underline{z}=0,$$
    so $nx+ky=nx'+ky'$. Therefore the $(-2)$-roots $\vect{r}$ and $\vect{r'}$ differ by a multiple of the primitive isotropic vector $\vect{v}$, and this shows that the homomorphism $(\vect{v}^\perp/\vect{v})_{root} \rightarrow W_{root}$ is injective. In particular $\vect{v}^\perp/\vect{v}$ is not a root overlattice.
    
    Now by assumption $\mathcal{D}_L$ contains a unique cusp with infinite stabilizer, namely $\vect{e}$, so by Proposition \ref{prop:generalization.Shioda.Tate} the vector $\vect{v}$ is not fundamental. Since $\vect{v}$ is positive, this implies that there exists a positive $(-2)$-root $\vect{r}=x\vect{e}+y\vect{f}+\vect{z}\in L$ such that $\vect{v}.\vect{r}<0$. Since $\vect{r}$ is positive, we have $y=\frac{1}{n}\vect{e}.\vect{r}>0$. 
    By Lemma \ref{lemma:computation} 
    $$\vect{v}.\vect{r} = \frac{n}{y}\left(-\frac{1}{2} \vect{z}^2 - 1\right),$$
    and since $\vect{v}.\vect{r} <0$, then necessarily $\vect{z}=\vect{0}$. Hence $-2 = \vect{r}^2 = 2nxy +2ky^2$, or equivalently $y(nx+ky)=-1$. It follows that $y=\pm 1$ and $nx+ky = nx\pm k = \mp 1$, that is $nx = \mp (k+1)$ and $n$ divides $k+1$.
    We deduce that $k\equiv -1 \pmod{n}$, and up to substituting $\vect{f}$ with $\vect{f}+\alpha\vect{e}$ for some $\alpha\in \ZZ$, we may assume that $k=-1$. 
    
    \item This follows from Lemma \ref{lemma:structure.lemma}.
    
    \item Consider the overlattice $M$ of $L$ spanned by $\{\vect{e}/n,\vect{f},\vect{w_1},\ldots,\vect{w_r}\}$. It is immediate to notice that the associated Gram matrix is as in (\ref{eq:matrix}) with $n=1$. As in point (2), we may assume up to isometry of $M$ that $\underline{\ell}=\underline{0}$, so $M$ is isometric to $U\oplus W$. Since by Proposition \ref{prop:entropy.overlattice} the hyperbolic lattice $M\cong U\oplus W$ has zero entropy, we conclude that $W$ is a Leech lattice.
    
    \item For any vector $\vect{w}\in W$ we can consider the basis $\{\vect{e},\vect{f}+\vect{w},\vect{w_1},\ldots,\vect{w_r}\}$ of $L$. The Gram matrix of $L$ with respect to this new basis is exactly as in (\ref{eq:matrix}), except for the value of $k$, which now equals $k'=\frac{1}{2}(\vect{f}+\vect{w})^2$. Reasoning as in point (1), we have that $k'\equiv -1 \pmod{n}$.
    
    Say $\vect{w_i}^2=-2k_i$. By choosing $\vect{w}=\pm \vect{w_i}$ we obtain
    $$\frac{1}{2}(\vect{f}\pm \vect{w_i})^2 \equiv -1 \pm \ell_i - k_i \equiv -1 \pmod{n},$$
    that is $\ell_i\equiv \pm k_i \pmod{n}$ for any $i$. Consequently $2k_i \equiv 0 \pmod{n}$, i.e. $n$ divides the diagonal entries of $W$.
    By choosing instead $\vect{w}=\vect{w_i}+\vect{w_j}$ we have similarly $\ell_i+\ell_j \equiv k_i + k_j - \vect{w_i}.\vect{w_j} \pmod{n}$, and therefore $\vect{w_i}.\vect{w_j} \equiv 0 \pmod{n}$. In other words, $n$ divides all the entries of the matrix $W$.
    \end{enumerate}
   
\end{proof}

Proposition \ref{prop:Gram.matrix.Borcherds.lattices} puts heavy restrictions on the Gram matrix of a Borcherds lattice $L$: indeed, 
if the scale of $W$ is $1$, then $L$ is isometric to $U\oplus L$.
Since we have already classified the Borcherds lattices containing a copy of $U$, we can assume that $W$ is a Leech type lattice of scale $>1$.

Among the root overlattices, there is only one Leech type lattice of scale $>1$, namely $A_1^9$, which has scale $2$. The unique lattice in its genus that is not a root overlattice is $E_8(2)\oplus A_1$. On the other hand, among the lattices unique in their genus there are $30$ Leech type lattices of scale $>1$, and Proposition \ref{prop:Gram.matrix.Borcherds.lattices} provides a straightforward strategy to classify the remaining Borcherds lattices, starting from these $31$ lattices.

Indeed let $W$ be one of the Leech type lattices of scale $c>1$. Following the notation of the matrix (\ref{eq:matrix}), by Proposition \ref{prop:Gram.matrix.Borcherds.lattices} we have that $k=-1$ and that $n$ is a divisor of $c$. Fix a divisor $n>1$ of $c$. Then again by Proposition \ref{prop:Gram.matrix.Borcherds.lattices} we just need to consider the $n^{\rk(W)}$ hyperbolic lattices with Gram matrix as in (\ref{eq:matrix}), corresponding to each possible vector $\underline{\ell}\in (\ZZ/n)^{\rk(W)}$, and decide which of them are Borcherds lattices.

We employ the following strategy to avoid unnecessary computations. After fixing $n>1$, many of the resulting $n^{\rk(W)}$ lattices are isometric. In order for two hyperbolic lattices $L_1$, $L_2$ to be isometric, it is sufficient that they are in the same genus and that $\ell(L_1)\le \rk(L_1)-2$ (since this last condition ensures that $L_1$ is unique in its genus by \cite[Corollary~1.13.3]{nikulin.integral.symmetric.bilinear}). For instance, in the case $W=E_8(2)\oplus A_1$ and $n=2$, there are only $5$ distinct genera corresponding to the different $\underline{\ell}\in (\ZZ/2)^{9}$, and if $\underline{\ell}\ne \underline{0}$, the length of the resulting hyperbolic lattice is $9$. This reduces the number of total hyperbolic lattices to consider from $2^9$ to $5$.

We apply Borcherds' method to decide whether the hyperbolic lattices resulting from the previous discussion are Borcherds lattices or not, and this completes the classification of Borcherds lattices.

\section{K3 surfaces of zero entropy} \label{sec:K3}

In this section we will apply the general results about hyperbolic lattices of zero entropy to the case of K3 surfaces. In the following $k=\overline{k}$ is an algebraically closed field of characteristic $p\ge 0$. 

A K3 surface is a smooth projective surface $X$ over $k$ with trivial canonical bundle $K_X=0$ and with $H^1(X,\mathcal{O}_X)=0$. The Picard group $\Pic(X)$ of $X$ is a finitely generated free $\ZZ$-module of rank $\rho(X)\le 20$ or $\rho(X)=22$, and by the Hodge index theorem it has the structure of a hyperbolic lattice. The rank $\rho(X)$ of the Picard group is called the \emph{Picard rank} of $X$. If the characteristic $p$ is zero, then $\rho(X)\le 20$ by Hodge theory, and $\Pic(X)$ admits a primitive embedding into the second cohomology group $H^2(X,\ZZ)$, which is an even unimodular lattice of signature $(3,19)$ \cite[Proposition~1.3.5]{huybrechts.K3}. In particular $H^2(X,\ZZ)$ is abstractly isometric to the lattice $U^3\oplus E_8^2$.
The K3 surfaces of Picard rank $22$, which can only exist in positive characteristic, are called \emph{supersingular}. 

For any automorphism $f\in \Aut(X)$ of the K3 surface $X$, we can consider its induced action $f^*$ on $L=\Pic(X)$, which naturally preserves the lattice structure on $\Pic(X)$ as well as the nef cone of $X$. The nef cone $\mathcal{D}_X=\mathcal{D}_L$ is a fundamental domain for the Weyl group. The homomorphism 
\[\Aut(X)\rightarrow \Aut(\mathcal{D}_L)\cong\Or^+(L)/W(L)\]
has finite kernel. Except for some supersingular K3 surfaces in characteristic $2$ and $3$ it is proven that it has finite cokernel too (see \cite[Theorem~15.2.6]{huybrechts.K3} for the case of characteristic $0$ and \cite[Theorem~6.1]{lieblich.maulik} for the case of odd characteristic). In this case the structure of the automorphism group of $X$ is determined up to finite index by the Picard lattice $L$. For instance, $\Aut(X)$ is finite (resp. virtually abelian or virtually solvable) if and only if the symmetry group $\Aut(\mathcal{D}_L)$ is finite (resp. virtually abelian or virtually solvable).

We define the \emph{entropy} $h(f)$ of an automorphism $f\in \Aut(X)$ as the entropy of the isometry $f^*\in \Aut(\mathcal{D}_L)$. Note that, if characteristic $p$ is zero, this definition coincides with the usual definition of entropy of an automorphism of a complex variety (cf. \cite[Théorème~2.1]{cantat.dynamics.K3.1} and the discussion in \cite{esnault.srinivas}). Since by Riemann-Roch the cusps of the nef cone correspond to genus one fibrations (i.e. elliptic or quasi-elliptic fibrations) on $X$, we have that an automorphism $f\in \Aut(X)$ has zero entropy if and only if either $f$ has finite order, or if $f$ preserves a genus one fibration  on $X$.

\begin{definition}
A K3 surface $X$ has \emph{zero entropy} if every automorphism of $X$ has zero entropy, or equivalently if every automorphism of infinite order preserves some genus one fibration on $X$. Otherwise we say that $X$ has \emph{positive entropy}.
\end{definition}
By definition a K3 surface has zero entropy if and only if its Picard lattice $\Pic(X)$ has zero entropy, or, equivalently by Theorem \ref{thm:classification.Borcherds.lattices}, if and only if either $X$ has finite automorphism group, or $\Pic(X)$ is a Borcherds lattice. Since the classification of K3 surfaces with finite automorphism group follows immediately from the classification of hyperbolic lattices with finite symmetry group due to Nikulin and Vinberg, we will assume in the rest of the section that the K3 surface $X$ has an infinite automorphism group.

K3 surfaces of zero entropy were previously studied by Nikulin in \cite{nikulin.preserves} and by the second author in \cite{mezzedimi.entropy}, where he obtained a partial classification of complex K3 surfaces of zero entropy. From our classification of Borcherds lattices we are now able to complete the classification of K3 surfaces of zero entropy in every characteristic.

We rephrase Theorem \ref{thm:classification.Borcherds.lattices} and Theorem\ref{thm:virtually.abelian} in the language of K3 surfaces. Recall that, if $|E|:X\rightarrow \PP^1$ is a genus one fibration on the K3 surface $X$, the Jacobian fibration $|JE|:JX\rightarrow \PP^1$ of $|E|$ is a \emph{Jacobian} genus one fibration (i.e. with a section) on another K3 surface $JX$. If $|E|$ already has a section, then $JX=X$ and $|JE|$ coincides with the genus one fibration $|E|$ itself. In any case, the stabilizer of $|E|$ in $\Aut(X)$ coincides up to a finite group with the Mordell-Weil group $\MW(JE)$ of the Jacobian fibration. We will call the rank of $\MW(JE)$ the \emph{Mordell-Weil rank} of the genus one fibration $|E|$.

\begin{theorem}[{cf. \cite[Theorem~1.6]{oguiso.entropy}}] \label{thm:equivalent.characterizations.K3}
Let $X$ be a K3 surface with infinite automorphism group. The following are equivalent:
\begin{enumerate}[(a)]
    \item $X$ has zero entropy;
    \item There exists a unique genus one fibration on $X$ whose Jacobian fibration has an infinite Mordell-Weil group;
    \item There exists a genus one fibration on $X$ preserved by all the automorphisms of $X$.
\end{enumerate}

Moreover, every K3 surface of zero entropy has a virtually abelian automorphism group. Conversely, every K3 surface with a virtually solvable automorphism group has zero entropy as soon as $\rho(X)\ge 5$.
\end{theorem}

Now Theorem \ref{thm:classification.Borcherds.lattices} provides a classification of K3 surfaces of zero entropy and infinite automorphism group, depending on their Picard lattice. We observe that all the Borcherds lattices, with the exception of $\mathrm{II}_{1,25}$, embed into the K3 lattice $U^3\oplus E_8^2$ by \cite[Corollary~1.12.3]{nikulin.integral.symmetric.bilinear}, since they all satisfy the condition $\rk(L)+\ell(A_L)\le 20$. Therefore the surjectivity of the period map ensures that there are K3 surfaces over $\CC$ with these Picard lattices. Their transcendental lattice can be easily computed as the orthogonal complement of $\Pic(X)$ in the K3 lattice.

Let us state the classification result for K3 surfaces of zero entropy. We will then provide some interesting consequences of this classification.

\begin{theorem} \label{thm:main.theorem.K3}
A K3 surface $X$ has zero entropy and infinite automorphism group if and only if its Picard lattice $\Pic(X)$ belongs to an explicit list of $193$ lattices.
\end{theorem}
\begin{proof}
The homomorphism $\varphi \colon \Aut(X) \to \Or^+(\Pic(X))/W(\Pic(X))$ has finite kernel in every characteristic. Therefore $X$ has zero entropy and infinite automorphism group if and only if $\Pic(X)$ is a Borcherds lattice. 
This does not rely on the Torelli theorem because the needed automorphisms are induced by the Mordell-Weil group of a genus one fibration.  
\end{proof}

We refer the interested reader to the ancillary file for the complete list of $193$ lattices. In Appendix B we include the Picard lattices of K3 surfaces of zero entropy and rank $\ge 11$.

%\begin{corollary} \label{cor:maximal.rank.zero.entropy}
%Let $X$ be a K3 surface with a virtually abelian automorphism group. Then $\Aut(X)$ contains a subgroup of finite index isomorphic to $\ZZ^m$, with $m\le 8$.
%\end{corollary}
%\begin{proof}
%If $X$ has positive entropy, then the rank of $\Aut(X)$ is $1$ by Proposition \ref{prop:virtually.abelian.positive.entropy}. If instead $X$ has zero entropy, the rank of $\Aut(X)$ can be computed via Proposition \ref{prop:generalization.Shioda.Tate}. In particular the rank is surely $\le 8$ if the Picard rank of $X$ is $\le 10$. If instead the Picard rank of $X$ is at least $11$, then $\Pic(X)$ belongs to the list in Table \ref{tab:Borcherds.lattices.11}, and $\Pic(X)=U\oplus W$ for some negative definite lattice $W$. Again by Proposition \ref{prop:generalization.Shioda.Tate} we have that the rank of $\Aut(X)$ is equal to $\rk(W')-\rk(W'_{root})$, where $W'$ is the unique lattice in the genus of $W$ that is not a root overlattice, and it is straightforward to check that $\rk(W')-\rk(W'_{root})\le 8$ in all the cases.
%\end{proof}

%Corollary \ref{cor:maximal.rank.zero.entropy} has the following interesting consequence: if $X$ is a K3 surface admitting a genus one fibration with Mordell-Weil rank $>8$, then $X$ has positive entropy. This criterion can be used in practice to decide whether a K3 surface with large Picard rank admits an automorphism of positive entropy without knowing explicitly the full Picard lattice.

\begin{corollary} \label{cor:examples.positive.entropy}
The following K3 surfaces have positive entropy, and in particular their automorphism group is not virtually solvable:
\begin{enumerate}
    \item Kummer surfaces in characteristic $0$ or $p>2$;
    \item K3 surfaces covering an Enriques surface, unless $\Pic(X)\cong U\oplus E_8\oplus D_8$;
    \item Singular and supersingular K3 surfaces.
    \item K3 surfaces with a genus one fibration with Mordell-Weil rank at least $9$.
\end{enumerate}
\end{corollary}
\begin{proof}
 (1) In characteristic $0$ a K3 surface $X$ is Kummer if and only if its transcendental lattice $\mathrm{T}(X)$ embeds primitively into the lattice $U(2)^3$ \cite[Theorem~14.3.17]{huybrechts.K3}. It is straightforward to check that none of the K3 surfaces of zero entropy are Kummer by computing their transcendental lattices. If instead $\mathrm{char}(k)=p>2$, either $X$ is supersingular, or it is liftable to a K3 surface $X_0$ in characteristic $0$ together with the full Picard group \cite[Corollary~4.2]{lieblich.maulik}.
 In the first case $X$ has positive entropy, because there are no $2$-reflective or Borcherds lattices of rank $22$. In the second $\Pic(X_0)= \Pic(X)$ contains $16$ orthogonal simple $(-2)$-roots, hence $X_0$ is Kummer as well and we conclude by the result in characteristic $0$.
 
(2) In any characteristic, if a K3 surface $X$ covers an Enriques surface, then there exists a primitive embedding $U(2)\oplus E_8(2)\hookrightarrow \Pic(X)$. If $X$ has zero entropy then it has either finite automorphism group or $\Pic(X)$ is a Borcherds lattice. It is straightforward, using the lists of $2$-reflective and Borcherds lattices, to check that $U(2)\oplus E_8(2)$ embeds primitively into such a lattice if and only if $\Pic(X)\cong U \oplus E_8\oplus D_8$.
    %\Si{\sout{Again we may assume that $X$ is not supersingular, so reasoning as above we have that $\Pic(X)$ embeds primitively in the K3 lattice. In particular the orthogonal $\mathrm{T}$ of $\Pic(X)$ in the K3 lattice (which is the transcendental lattice $\mathrm{T}(X)$ if the characteristic is $0$) embeds primitively into the orthogonal of $U(2)\oplus E_8(2)$ into the K3 lattice, which is isometric to $U\oplus U(2)\oplus E_8(2)$. Again it is straightforward to check that, if $\Pic(X)\not \cong U\oplus E_8\oplus D_8$, an embedding $\mathrm{T}(X)\hookrightarrow U\oplus U(2)\oplus E_8(2)$ cannot exist. For the case $\Pic(X)\cong U\oplus E_8\oplus D_8$, see Remark \ref{rk:U.E8.D8}.}}
    %\Si{I think that this argument is too complicated. It is more straightforward to just check that $U(2)\oplus E_8(2)$ does not embedd primitively in any Borcherds lattice.}
    
   (3) All hyperbolic lattices of rank $20$ and $22$ have an infinite symmetry group by Nikulin's classification, and there are no Borcherds lattices of rank $20$ or $22$.
   
    (4) Assume by contradiction that $X$ has zero entropy. Then $X$ has an infinite virtually abelian automorphism group. By assumption $\rho(X)\ge 11$, so $\Pic(X)=U\oplus W$ is a Borcherds lattice by Theorem \ref{thm:virtually.abelian}. By assumption $\Aut(X)$ contains an abelian subgroup of rank at least $9$. 
    %, and by Proposition \ref{prop:generalization.Shioda.Tate} its rank is equal to $\rk(W')-\rk(W'_{root})$, where $W'$ is the unique lattice in the genus of $W$ that is not a root overlattice. 
    But by Table \ref{tab:Borcherds.lattices.11} it is at most 8, a contradiction.
\end{proof}

It was already proved by Oguiso in \cite[Theorem~1.6]{oguiso.entropy} that singular K3 surfaces over $\CC$ have positive entropy. The same was shown for supersingular K3 surfaces in \cite{yu.salem} and \cite{brandhorst.22}.
Corollary \ref{cor:examples.positive.entropy} can be used in practice to determine whether a K3 surface with large Picard rank admits an automorphism of positive entropy, without knowing explicitly the full Picard lattice.

\begin{remark} \label{rk:U.E8.D8}
Let us explain the geometry of complex K3 surfaces $X$ with Picard lattice isometric to $U\oplus E_8\oplus D_8$. Similar results could be proved over algebraically closed fields of arbitrary characteristic.

There exists a unique elliptic fibration $|E|$ on $X$ with Mordell-Weil group of positive rank, which admits a unique reducible fiber of type $\mathrm{I}_{16}$. In particular the Mordell-Weil group of $|E|$ has rank $1$, and by Theorem \ref{thm:equivalence.characterizations.Borcherds} it follows that $\Aut(X)\cong \ZZ$ up to a finite group. It was already proved by Nikulin \cite[{\S}6]{nikulin.preserves} that such K3 surfaces have zero entropy, using the following observation. Since $\Pic(X)$ is $2$-elementary, $X$ admits a non-symplectic involution $\sigma$, and we can study its fixed locus. It follows from \cite[Equation~(5)]{nikulin.preserves} that the fixed locus contains a curve $C$ of genus $1$, and since the whole automorphism group $\Aut(X)$ commutes with $\sigma$, the whole $\Aut(X)$ must fix the class of $C$ in $\Pic(X)$. In particular $X$ has zero entropy, and the fixed curve $C$ is a fiber in the unique elliptic fibration $|E|$ with positive Mordell-Weil rank.

We can also explicitly describe which Enriques surfaces are covered by $X$.
One can show that $X$ covers a unique Enriques surface $S$ up to isomorphism. More precisely $S$ is a general member of the $2$-dimensional family studied by Barth and Peters (see for instance \cite[Lemma~4.13]{barth.peters.enriques}). Barth and Peters studied the Enriques surfaces in this family as examples of Enriques surfaces with an infinite, but virtually abelian automorphism group. In fact it turns out that the automorphism group of $S$ has a subgroup of finite index isomorphic to $\ZZ$ (cf. \cite[Theorem~4.12]{barth.peters.enriques}).

The Enriques surfaces in the Barth-Peters family can be characterized by the fact that their dual graph of $(-2)$-curves contains the following graph: 

$$
    
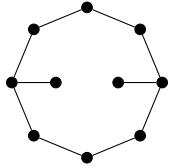
\begin{tikzpicture}[scale=0.5]
\node (R1) at (180:2) [nodal] {};
\node (R2) at (135:2) [nodal] {};
\node (R3) at (90:2) [nodal] {};
\node (R4) at (45:2) [nodal] {};
\node (R5) at (0:2) [nodal] {};
\node (R6) at (315:2) [nodal] {};
\node (R7) at (270:2) [nodal] {};
\node (R8) at (225:2) [nodal] {};
\node (R9) at (intersection of R2--R7 and R3--R8) [nodal] {};
\node (R10) at (intersection of R4--R7 and R3--R6) [nodal] {};

\draw (R1)--(R2)--(R3)--(R4)--(R5)--(R6)--(R7)--(R8)--(R1)--(R9) (R5)--(R10);
    \end{tikzpicture}
$$

Note that the half-fiber of type $\mathrm{I}_8$ on $S$ pulls back to the $\mathrm{I}_{16}$ fiber on $X$. Moreover the automorphism group of $S$ preserves the half-fiber of type $\mathrm{I}_8$ \cite[Lemma~4.10]{barth.peters.enriques}.

Finally, let us observe that the Enriques surfaces in the Barth-Peters family are special from several points of view: not only they are one of the few families of Enriques surfaces admitting a numerically trivial automorphism \cite{mukai.numerically.trivial}, but they also are the only Enriques surfaces in characteristic $\ne 2$ admitting a non-extendable $3$-sequence \cite[Theorem~1.3]{martin.mezzedimi.veniani.3}.
\end{remark}

\section{Appendix A: Borcherds' method} \label{sec:borcherds.method}
We review Borcherds' method, which is an algorithm that computes the symmetry group of an arbitrary hyperbolic lattice $S$ embedding into $\mathrm{II}_{1,25}$, up to a finite group. This is enough to decide whether the lattice $S$ has zero entropy. For details and proofs we refer to \cite{shimada.borcherds}.

\subsection{Conway chambers}
Recall that $\Lambda$ denotes the Leech lattice. We set $L = U \oplus \Lambda$ and call any fundamental domain for the Weyl group of $L$ a \emph{Conway chamber} and denote it by $C$. For instance $\mathcal{D}_L$ is a Conway chamber. Note that $C$ is a locally polyhedral convex cone. 

\subsection{Weyl vectors}
For a lattice $N$, $\Delta_N = \{\vect{r} \in N : \vect{r}^2 = -2\}$ denotes the set of $(-2)$-roots. A $(-2)$-root $\vect{r}$ defines a half space $H_{\vect{r}}=\{\vect{x} \in \mathcal{P}_N  : \vect{x}.\vect{r}\geq 0\}$.
We call $\vect{w} \in L$ a \emph{Weyl vector} (of the Conway chamber $C$) if the set of simple $(-2)$-roots of $L$ (with respect to $C$) coincides with $\Delta(C) = \{\vect{x} \in \Delta_L : \vect{w}.\vect{x} = 1\}$.
Recall that the simple $(-2)$-roots are in bijection with the facets of $C$.

Conway \cite[Ch.~27,~{\S}2,~Theorem~1]{conway.sloane.lattices} proved that every Conway chamber has a unique Weyl vector. More precisely, $\vect{w}$ is a Weyl vector of $L$ \mbox{ if and only if } $\vect{w}^2 = 0$ and $\vect{w}^\perp / \langle \vect{w} \rangle \cong \Lambda$. 
He also showed that the group of symmetries $\Aut(C)$ is isomorphic to the affine group of the Leech lattice. In particular $\Aut(C)$ is virtually abelian of rank $24$.

\subsection{Induced Conway Chambers}
Borcherds' method uses our detailed knowledge of $L$ to 
compute a finite index subgroup of $\Aut(\mathcal{D}_S)$, where $S \subseteq L$ is any primitive sublattice such that $R = S^\perp \subseteq L$ cannot be embedded in the Leech lattice. 

This condition is true if for instance $R$ contains at least a $(-2)$-root. 
In this case $C_S \coloneqq  C\cap S$ lies in a face of $C$.
Since $C$ is locally polyhedral, the chamber $C_S$ is actually a \emph{finite} polyhedral cone. It may happen that $\dim C_S < \rk S$; in this case we call the chamber $C$ and its Weyl vector \emph{$S$-degenerate}. By a suitable choice of $C$, we can always ensure that $C_S$ contains an open subset of $\mathcal{D}_S$. 

Since $\Delta_S \subseteq \Delta_L$, we have that $C_S \subseteq \mathcal{D}_S$ for a unique fundamental chamber $\mathcal{D}_S$ of $S$.
Furthermore, we know that the Conway chambers tile the positive cone of $L$. 
Since we can see the positive cone of $S$ as a slice of the positive cone of $L$, the tessellation of $\mathcal{P}_L$ by Conway chambers $C$ 
induces a tessellation of $\mathcal{P}_L \cap S_\RR = \mathcal{P}_S$ by induced Conway chambers $C_S$. The dual picture is as follows:
let $\pi: L_\RR \to S_\RR$ be the orthogonal projection. 
Set $\Delta_{L|S} = \pi(\Delta_L) \setminus \{0\} \subseteq S \otimes \QQ$. Then the tessellation by induced chambers has walls defined by $\Delta_{L|S}$ and $\Delta_S \subseteq \Delta_{L|S}$. There are two types of walls: the elements of $\Delta_S$ are called \emph{outer walls} and the elements of $\Delta_{L|S} \setminus \Delta_{S}$ \emph{inner walls}. 

\subsection{Adjacent Chambers}
We call two induced chambers $\gamma_1$ and $\gamma_2$ \emph{adjacent}, if they share a facet. This facet is cut out by a wall $\vect{v} \in \Delta_{L|S}$. Suppose $\gamma_1 \subseteq \mathcal{D}_S$.
If $\vect{v}$ is an inner wall, then $\gamma_2 \subseteq \mathcal{D}_S$ as well, while if $\vect{v}$ is an outer wall, then $\gamma_2$ is not contained in $\mathcal{D}_S$, but rather in the mirrored Weyl chamber $s_{\vect{v}}(\mathcal{D}_L)$.

\subsection{The chamber graph}
Define an infinite graph $\Gamma$ with vertices given by the set $\mathcal{C}_S$ of induced Conway chambers. Two chambers $\gamma_1$ and $\gamma_2$ are joined by an edge if and only if they are adjacent by a wall. Recall that the set of edges emanating from a given vertex is finite, since $\gamma$ is a finite polyhedral cone.

An isometry $f \in \Or^+(S)$ preserves the Weyl chambers of $S$, but it may not preserve the tessellation of the Weyl chambers by induced Conway chambers. 
A solution is to pass to the finite index subgroup $G \subseteq \Or^+(S)$ consisting of those elements of $\Or^+(S)$
that extend to an isometry of $L$. Clearly the isometries of $L$ preserving $S$ map induced Conway chambers to induced Conway chambers. Therefore the group $G$ acts on $\Gamma$, and it is known that $\Gamma/G$ is finite. 
We call two chambers in $\Gamma$ \emph{$G$-congruent} if they lie in the same $G$-orbit.
We set $\Hom_G(\gamma_1,\gamma_2)=\{g \in G : g(\gamma_1) = \gamma_2\}$.

\subsection{Borcherds' method - Shimada's algorithm}
To work with $\Gamma/G$, we rely on algorithms computing the following:
\begin{enumerate}
    \item given $\gamma \in \Gamma$, return the finite list of $\gamma'\in \Gamma$ sharing an edge with $\gamma$;
    \item given two vertices $\gamma_1,\gamma_2$, compute the finite set $\Hom_G(\gamma_1,\gamma_2)$.
\end{enumerate}
Note that (2) allows to decide whether or not $\gamma_1$ and $\gamma_2$ are $G$-congruent.
Then $\Gamma/G$, as well as generators for $G$,
can be computed by a standard algorithm in geometric group theory. 
At the heart is the computation of a spanning tree in the finite graph $\Gamma/G$. 
We obtain a new generator $g$ for the group $G$ whenever we encounter an ``unexplored'' chamber $\gamma_1$ which is $G$-congruent to an already ``explored'' chamber $\gamma_2$
or an unexplored chamber with $\Hom_G(\gamma,\gamma)=:\Aut_G(\gamma)$ trivial.

Note that given an edge, i.e. a wall, it is easy to decide if it is an inner or outer wall.
Therefore we can work in the subgraph 
$\Gamma(\mathcal{D}_S) =\{ \gamma \in \Gamma : \gamma \subseteq \mathcal{D}_S\}$ and use the group $\Aut_G(\mathcal{D}_S) = G \cap \Aut(\mathcal{D}_S)$ in place of $G$. 
Note that $\Gamma/G \cong  \Gamma(\mathcal{D}_S)/\Aut_G(\mathcal{D}_S)$.

The input of Shimada's algorithm consists of the triple $(L,S,\vect{w})$, where $\vect{w}$ is a suitable Weyl vector of $L$.
The output consists of generators for $\Aut_G(\mathcal{D}_S)$, as well as a list of Conway chambers in $\Gamma(\mathcal{D}_S)$ constituting a complete set of representatives of $\Gamma/G$.
Along the way it also computes a set of representatives of the simple $(-2)$-roots, i.e. outer walls $\Delta(\mathcal{D}_S) / \Aut_G(\mathcal{D}_S)$.

Note that $\Aut_G(\mathcal{D}_S)$ is of finite index in $\Aut(\mathcal{D}_S)$. Therefore this suffices for our purpose of determining whether $S$ has zero entropy or not.

\subsection{Complexity}
The complexity of this algorithm can be estimated roughly as follows:
let $\vect{v_1},\dots, \vect{v_n}$ be the vertices of $\Gamma$ that we have already explored.
Then, for each new vertex $\vect{v} \in \Gamma$ one has to check whether there is an $i \in I$ and a $g\in G$ with $g(\vect{v_i})=\vect{v}$.
This leads to a worst case of $n$ checks for each new vector and leads to a time complexity of roughly $c n(n+1)/2$, where $c$ is the time needed to compute $\Hom_G(\gamma, \gamma')$ for $\gamma, \gamma' \in \Gamma$. 

In the largest example that we computed, $n$ is of magnitude $5\cdot 10^6$, leading to a time complexity of $10^{13}$, which is by far too big for a practical algorithm.
In what follows we report on our improvements to Shimada's algorithm.

The complexity can be decreased to (very roughly) $2c n$ if one finds invariants of the vertices separating the $G$-orbits; then one has to perform at most a single check per new vertex $\gamma$ and compute $\Aut_G(\gamma)$. 
Finding invariants separating the $G$-orbits is too much to ask for, but any invariant separating ``most'' $G$-orbits leads to a drastic speedup. The fingerprint is one such invariant. 

\subsection{The fingerprint of a chamber}
Let $\gamma \in \Gamma$ be an induced Conway chamber. A facet of $\gamma$ corresponds to a ray $F = \RR_{\geq 0} \vect{v}$ of its dual cone, where $\vect{v} \in \Delta_{L|S}$. Then $F \cap S^\vee = \NN_0 \vect{v'}$; we call $\vect{v'}$ a \emph{primitive facet generator} of $\gamma$. 

Let $(\vect{v_1}, \dots, \vect{v_n}) $ be the primitive facet generators of the chamber $\gamma$. 
Set $\vect{a}=\sum_{i=1}^n\vect{v_i}$ and $a_\gamma = \vect{a}^2$. For $i \in \{1,\dots n\}$, set $b_i =\vect{a}^2$ and
$c_i = (\vect{v_i}.\vect{a} : i \in \{1,\dots n\})$. 
Let $b_\gamma$ (resp. $c_\gamma$) 
be the list of $b_i$ (resp. $c_i$) with entries sorted in ascending order. 
The \emph{fingerprint} of the induced Conway chamber $\gamma$ is the triple $f(\gamma) = (a_\gamma, b_\gamma, c_\gamma)$. 
By construction we have the following:

\begin{proposition}  
If $\gamma$ and $\gamma'$ are $G$-congruent, then they have the same fingerprint. 
\end{proposition}

The reader may notice that the definition of the fingerprint does not involve $G$; it is more of an invariant for $\Or^+(S)$ than $G$.
If the index $[\Or^+(S):G]$ is large, it can be worth refining the fingerprint by using the $G$-orbits on the discriminant group $S^\vee/S$.
In general the fingerprint is not enough to separate all $G$-orbits, but in practice it separates most of them. 

\subsection{Checking G-congruence.}
Given the primitive facet generators $\Delta_1$ and $\Delta_2$ of the induced Conway chambers $\gamma_1$ and $\gamma_2$, we can compute the set 
$\Hom_G(\gamma_1,\gamma_2)$ as follows.
Notice that $\Hom_G(\gamma_1, \gamma_2) =\{g \in \Or^+(S) : g(\Delta_1) = \Delta_2, g\in G\}$. Since $\gamma_i \subseteq S \otimes \RR$ has full dimension, we can choose a basis $\vect{b_1}, \dots, \vect{b_\rho} \in \Delta_1$ of $S\otimes \QQ$. 
If $g\in \Hom_G(\gamma_1,\gamma_2)$, then we know that $g(\vect{b_i}) \in \Delta_2$, and since $\Delta_2$ is finite, this shows that $\Hom_G(\gamma_1,\gamma_2)$ is finite.
Conversely, in order to obtain an element of $\Hom_G(\gamma_1, \gamma_2)$, we choose $\rho$ elements $\vect{v_1},\dots, \vect{v}_\rho \in \Delta_2$ and define $g \in \GL(S \otimes \QQ)$ by $g(\vect{b_i})=\vect{v_i}$.  
Then one checks if $g \in \Or(S)$, $g(\Delta_1) = \Delta_2$ and finally if $g \in G$.
Shimada proceeds by brute force and enumerates $\Delta_2^\rho$ to filter out $\Hom_G(\gamma_1,\gamma_2)$. This works well if $(\#\Delta_2)^\rho$ is small. 

For a more efficient approach, we rely on the ideas presented in \cite{plesken.souvignier}. Originally their algorithm computes isometries between two positive definite lattices $W_1$ and $W_2$. 
It can be modified to instead compute  $\Hom_G(\gamma_1,\gamma_2)$. The idea is to replace the finite set of short (enough) vectors of $W_i$ with the finite set $\Delta_i$ of primitive facet generators. 
Anything else is straightforward and left to the reader. 

\subsection{Computing the facets}
From the Weyl vector $\vect{w}$ of a Conway chamber $C$, Shimada computes the finite set  $\pi(\Delta(C)) \subseteq \Delta_{L|S}$ by enumerating solutions to an inhomogeneous quadratic equation $\underline{x}^T Q \underline{x} + 2 \underline{b}^T \underline{x} + c \leq 0$, where $Q \in \ZZ^{\rho \times \rho}$ is a positive definite matrix and $\underline{b} \in \ZZ^\rho$. 
For this enumeration Shimada refers to his Algorithm 3.1 on ``positive quadratic triples'' in \cite{shimada.char5}.
We remark that completing the square makes this equivalent to a close vector enumeration. The close vector problem is NP hard and well studied, and a fast algorithm for close and short vector enumeration is given for instance in \cite{fincke.pohst}. 
Finally, we would like to mention that it is even possible to adapt Shimada's Algorithm 5.8 in such a way as to just rely on a suitable short vector enumeration which leads to a further speedup.

The set $\pi(\Delta(C))$ thus computed is finite, and the induced chamber is given by $C_S= \{\vect{x} \in \mathcal{P}_S : \vect{x}.\vect{r} \geq 0 \  \forall r \in \pi(\Delta(C))\}$. Note that $\pi(\Delta(C))$ does not necessarily correspond to the set of walls of $C_S$, since some of the corresponding inequalities may be redundant. It is a standard task in algorithmic convex geometry to get rid of the redundancies. The algorithms can be based on linear programming for instance.
This gives the facets of $C_S$ and hence the edges of the graph $\Gamma$ adjacent to $\gamma = C_S$, as well as the primitive facet generators.

\begin{remark}
In higher dimension, getting rid of the redundancies is the bottleneck of the algorithm. To reduce the number of redundancy computations one can work with $\pi(\Delta(C))$ instead of the primitive facet generators because it is compatible with $G$-congruence. See \cite[Remark 6.8]{shimada.borcherds}. This is possible for the fingerprint and for checking $G$-congruence. Then the computation of the facets is only necessary for determining the edges of the graph.
\end{remark}

\subsection{Computing the first Weyl vector}
Given $S$ of rank $\rho$, one can compute a representative $R$ in the genus with discriminant form given by $-q|_{A_S}$ and signature $(0, 26-\rho)$.
Then $L$ is constructed as a primitive extension of $S \oplus R$ using an anti-isometry of the discriminant forms of $S$ and $R$.
The lattice $L$ thus obtained is even, unimodular and of signature $(1,25)$ hence it is \emph{abstractly} isomorphic to $U \oplus \Lambda$. 

To find a first Weyl vector, Shimada seems to rely on a random search of isotropic vectors in $L$. 
Here we give an algorithm using the $23$ holy constructions of the Leech lattice. 
At the heart is an algorithm which constructs an explicit isometry $L \cong U\oplus \Lambda$. Since the lattices involved are indefinite, this is hard in general.

First of all Simon's indefinite LLL-algorithm \cite{simon.lll} gives us a hyperbolic plane $U \subseteq L$. Then
we have $L = U \oplus N$ for some even negative definite unimodular lattice $N$.
If $N$ is the Leech lattice, we are done. Otherwise $N$ is one of the $23$ Niemeier lattices, corresponding to the $23$ deep holes of the Leech lattice. 
From this correspondence one infers $23$ constructions of the Leech lattice, one from each Niemeier lattice. For the details we refer to \cite[Theorem 4.4]{ebeling.lattices} and 
\cite[Chapter 24]{conway.sloane.lattices}. 

The outcome is a copy of $\Lambda$ in $N \otimes \QQ$ with 
\[N/(N \cap \Lambda) \cong \Lambda / (N \cap \Lambda) \cong \ZZ/ h \ZZ,\] 
where $h$ is the (common) Coxeter number of (the irreducible root sublattices) of $N$. In fact $\Lambda$ is constructed from a certain $[\vect{v}]\in N/hN$ as follows: set 
$$K_{\vect{v}} = \{\vect{x} \in N : \vect{x}.\vect{v} \equiv 0 \mod h\} \mbox{ and } \Lambda\coloneqq K_{\vect{v}}+(1/h)\vect{v}$$
for a representative $\vect{v}$ of $[\vect{v}]$ with $\vect{v}^2$ divisible by $2h^2$. Note that $\Lambda \cap N = K_{\vect{v}}$. 

We can use this $\vect{v}$ and the hyperbolic plane $U$ to construct an explicit isometry $U \oplus N \cong U \oplus \Lambda$ as follows.

\begin{theorem}
Choose a basis $\vect{e},\vect{f} \in U$ with $\vect{e}^2 = \vect{f}^2=0$ and $\vect{e}.\vect{f}=1$, and define $\vect{w} = -\vect{v}^2/(2h)\vect{f}+h \vect{e} + \vect{v}$. 
Then $\vect{w}$ is a Weyl vector, i.e. $\vect{w}^\perp/\langle \vect{w} \rangle \cong \Lambda$.
\end{theorem}
\begin{proof}
The proof in \cite[\S 2.1]{brandhorst.elkies} can be adapted to non-prime numbers.
\end{proof}

\subsection{A non-degenerate Weyl vector}
Recall that we need the first Weyl vector $\vect{w}$, with associated chamber $C$, to be $S$-nondegenerate. 
It is $S$-degenerate if $C_S \coloneqq  C \cap \mathcal{P}_S$ is not of the same dimension as $S$. If in the previous step we obtain an $S$-degenerate Weyl vector, we proceed as follows.
Let $N\coloneqq  C_S^\perp \subseteq L$ and $R = S^\perp \subseteq L$.
Choose a (random) fundamental vector $\vect{a} \in \mathcal{P}_S\setminus \bigcup_{\vect{r} \in \Delta_N\setminus \Delta_R} \vect{r}^\perp$, preferably close to $C_S$, and let $\Delta(\vect{w},\vect{a})\coloneqq \{\vect{r} \in \Delta_N \setminus \Delta_R : \vect{r}.\vect{w}>0 , \vect{r}.\vect{a}<0\}$ be the set of relevant roots. 
We sort $\{\vect{r_1},\dots, \vect{r_N}\}=\Delta(\vect{w},\vect{a})$ in a way so that 
\[i<j \implies \frac{\vect{u}.\vect{r_i}}{ \vect{a}.\vect{r_i}} < \frac{\vect{u}.\vect{r_j} }{ \vect{a}.\vect{r_j}},\]
where $\vect{u}$ is a general enough element of $C$. We set $l_i = s_{\vect{r_i}}$ and observe that $l_N \circ \ldots \circ l_1 (\vect{w})$ is a non-$S$-degenerate Weyl vector.

\clearpage

\section{Appendix B: Tables}

\small

\begin{table}[H]
\centering
\begin{tabular}{ c | c } 
  Rank & Lattice\\
  \hline
  \hline
\multirow{1}{1.5em}{$17$}
&$E_8^2\oplus A_1$\\
\hline
\multirow{1}{1.5em}{$16$}
&$E_8^2$\\
\hline
\multirow{1}{1.5em}{$15$}
&$E_8\oplus E_7$\\
\hline
\multirow{1}{1.5em}{$14$} 
&$E_8\oplus D_6$ \\
  \hline
\multirow{1}{1.5em}{$13$} 
&$E_8\oplus D_4\oplus A_1$ \\
  \hline
\multirow{3}{1.5em}{$12$}
&$E_8\oplus D_4$ \\
&$E_8\oplus A_1^4$ \\
&$D_8\oplus D_4$\\
  \hline
  \multirow{3}{1.5em}{$11$}
&$E_8\oplus A_3$ \\
&$E_8\oplus A_1^3$\\
&$E_7\oplus A_1^4$\\
  \hline
\multirow{4}{1.5em}{$10$}
&$E_8\oplus A_1^2$\\
&$E_8\oplus A_2$\\
&$E_7\oplus A_1^3$\\
&$D_6\oplus A_1^4$ \\
  \hline
\multirow{4}{1.5em}{$9$} 
&$E_7\oplus A_1^2$\\
&$E_8\oplus A_1$\\
&$D_6\oplus A_1^3$\\
&$D_4\oplus A_1^5$\\
\end{tabular}
\quad
\begin{tabular}{ c | c } 
  Rank & Lattice\\
  \hline
  \hline
\multirow{9}{1.5em}{$8$} 
&$D_8$\\
&$E_8$\\
&$E_7\oplus A_1$\\
&$E_6\oplus A_2$\\
&$D_6\oplus A_1^2$\\
&$D_4^2$\\
&$D_4\oplus A_1^4$\\
&$A_1^8$\\
&$O(A_1^8,2)$\\
  \hline
\multirow{9}{1.5em}{$7$} 
&$A_7$\\
&$D_7$\\
&$E_7$\\
&$D_6\oplus A_1$\\
&$E_6\oplus A_1$\\
&$D_5\oplus A_2$\\
&$D_4\oplus A_3$\\
&$D_4\oplus A_1^3$\\
&$A_1^7$\\
\end{tabular}
\quad
\begin{tabular}{ c | c } 
  Rank & Lattice\\
  \hline
  \hline
\multirow{11}{1.5em}{$6$} 
&$A_6$\\
&$D_6$\\
&$E_6$\\
&$A_5\oplus A_1$\\
&$D_5\oplus A_1$\\
&$A_4\oplus A_2$\\
&$D_4\oplus A_2$\\
&$D_4\oplus A_1^2$\\
&$A_3^2$\\
&$A_2^3$\\
&$A_1^6$\\
\hline
\multirow{8}{1.5em}{$5$}
&$A_5$\\
&$D_5$\\
&$A_4\oplus A_1$\\
&$D_4\oplus A_1$\\
&$A_3\oplus A_2$\\
&$A_3\oplus A_1^2$\\
&$A_2^2\oplus A_1$\\
&$A_1^5$\\
\end{tabular}
\begin{tabular}{ c | c } 
  Rank & Lattice\\
  \hline
  \hline
\multirow{6}{1.5em}{$4$}
&$A_4$\\
&$D_4$\\
&$A_3\oplus A_1$\\
&$A_2^2$\\
&$A_2\oplus A_1^2$\\
&$A_1^4$\\
\hline
\multirow{3}{1.5em}{$3$}
&$A_3$\\
&$A_2\oplus A_1$\\
&$A_1^3$\\
\hline
\multirow{2}{1.5em}{$2$}
&$A_2$\\
&$A_1^2$\\
\hline
\multirow{1}{1.5em}{$1$}
&$A_1$\\
\end{tabular}
\caption{List of $2$-reflective root overlattices. The notation $O=O(R,n)$ indicates that $O$ is a certain overlattice of $R$ of index $n$.}
\label{tab:2.reflective.root.overlattices}
\end{table}

\begin{table}[H]
\centering
\begin{tabular}{ c | c } 
  Rank & Lattice\\
  \hline
  \hline
\multirow{1}{1.5em}{$24$}
&$E_8^3$\\
\hline
\multirow{2}{1.5em}{$16$} 
&$E_8\oplus D_8$ \\
&$E_8\oplus E_7\oplus A_1$ \\
  \hline
\multirow{2}{1.5em}{$15$} &$E_7\oplus D_8$ \\
&$E_8\oplus D_7$ \\
  \hline
\multirow{2}{1.5em}{$14$} &$D_8\oplus D_6$ \\
&$E_8\oplus E_6$ \\
  \hline
  \multirow{3}{1.5em}{$13$} &$D_{10}\oplus A_1^3$ \\
&$E_7\oplus E_6$ \\
&$E_8\oplus D_5$ \\
  \hline
\multirow{5}{1.5em}{$12$}
&$D_8\oplus A_1^4$ \\
&$D_4^3$ \\
&$E_6^2$ \\
&$D_{11}\oplus A_1$ \\
&$E_8\oplus A_4$ \\
  \hline
\multirow{3}{1.5em}{$11$}
&$D_4^2\oplus A_1^3$ \\
&$E_8\oplus A_2\oplus A_1$ \\
&$D_7\oplus D_4$ \\
\end{tabular}
\quad
\begin{tabular}{ c | c | c } 
  Rank & Lattice &$\rho^2$\\
  \hline
  \hline
\multirow{5}{1.5em}{$10$}
&$D_4\oplus A_1^6$ &$5/2$\\
&$D_8\oplus A_2$ &$2$\\
&$E_6\oplus A_2^2$ &$2$\\
&$D_9\oplus A_1$ &$5/2$\\
&$D_7\oplus A_3$ &$2$\\
  \hline
\multirow{5}{1.5em}{$9$} &$A_1^{9}$ &$5/2$\\
&$E_7\oplus A_2$ &$5/2$\\
&$E_6\oplus A_2\oplus A_1$ &$5/2$\\
&$D_9$ &$2$\\
&$D_7\oplus A_2$ &$2$\\
  \hline
\multirow{7}{1.5em}{$8$} 
&$E_6\oplus A_1^2$ &$5/2$\\
&$A_2^4$ &$2$\\
&$D_7\oplus A_1$ &$5/2$\\
&$D_5\oplus A_3$ &$11/4$\\
&$A_4^2$ &$2$\\
&$A_7\oplus A_1$ &$5/2$\\
&$A_8$ &$2$\\
\end{tabular}
\quad
\begin{tabular}{ c | c | c } 
  Rank & Lattice &$\rho^2$\\
  \hline
  \hline
\multirow{7}{1.5em}{$7$}
&$A_2^3\oplus A_1$ &$5/2$\\
&$D_5\oplus A_1^2$ &$5/2$\\
&$A_4\oplus A_2\oplus A_1$ &$5/2$\\
&$A_4\oplus A_3$ &$11/4$\\
&$A_5\oplus A_1^2$ &$5/2$\\
&$A_5\oplus A_2$ &$17/6$\\
&$A_6\oplus A_1$ &$5/2$\\
\hline
\multirow{4}{1.5em}{$6$}
&$A_2^2\oplus A_1^2$ &$5/2$\\
&$A_3\oplus A_1^3$ &$5/2$\\
&$A_3\oplus A_2\oplus A_1$ &$11/4$\\
&$A_4\oplus A_1^2$ &$5/2$\\
\hline
\multirow{1}{1.5em}{$5$}
&$A_2\oplus A_1^3$ &$5/2$\\
\end{tabular}
\caption{Genus representatives of root overlattices of Leech type. The column $\rho^2$ indicates the square of the covering radius of the unique non-root overlattice in the genus.}
\label{tab:root.overlattices.Leech.type}
\end{table}

\begin{table}
\centering
\begin{tabular}{ c | c | c} 
  Rank &Lattice $L$ &$\rk(\Aut(\mathcal{D}_L))$ \\
  \hline
  \hline
  \multirow{1}{1.5em}{$26$} 
&$U\oplus E_8^3$ &$24$\\
  \hline
\multirow{2}{1.5em}{$18$} 
&$U\oplus E_8\oplus D_8$ &$1$ \\
&$U\oplus E_8\oplus E_7\oplus A_1$ &$1$ \\
  \hline
\multirow{2}{1.5em}{$17$}
&$U\oplus E_7\oplus D_8$ &$2$ \\
&$U\oplus E_8\oplus D_7$ &$1$\\
  \hline
\multirow{2}{1.5em}{$16$}
&$U\oplus D_8\oplus D_6$ &$3$ \\
&$U\oplus E_8\oplus E_6$ &$1$ \\
  \hline
  \multirow{3}{1.5em}{$15$} 
  &$U\oplus D_{10}\oplus A_1^3$ &$4$ \\
&$U\oplus E_7\oplus E_6$ &$1$ \\
&$U\oplus E_8\oplus D_5$ &$1$ \\
\hline
\multirow{5}{1.5em}{$14$}
&$U\oplus D_4^3$ &$6$ \\
&$U\oplus D_8\oplus A_1^4$ &$5$ \\
&$U\oplus E_6^2$ &$2$ \\
&$U\oplus D_{11}\oplus A_1$ &$1$ \\
&$U\oplus E_8\oplus A_4$ &$1$ \\
\end{tabular}
\quad
\begin{tabular}{ c | c | c } 
  Rank &Lattice $L$ &$\rk(\Aut(\mathcal{D}_L))$ \\
  \hline
  \hline
\multirow{3}{1.5em}{$13$}
&$U\oplus D_4^2\oplus A_1^3$ &$6$\\
&$U\oplus E_8\oplus A_2\oplus A_1$ &$1$\\
&$U\oplus D_7\oplus D_4$ &$1$\\
\hline
\multirow{5}{1.5em}{$12$}
&$U\oplus D_4\oplus A_1^6$ &$7$ \\
&$U\oplus E_6\oplus A_2^2$ &$3$\\
&$U\oplus D_7\oplus A_3$ &$2$ \\
&$U\oplus D_8\oplus A_2$ &$1$ \\
&$U\oplus D_9\oplus A_1$ &$1$\\
  \hline
\multirow{6}{1.5em}{$11$} 
&$U\oplus A_1^{9}$ &$8$\\
&$U\oplus E_6\oplus A_2\oplus A_1$ &$2$\\
&$U\oplus D_7\oplus A_2$ &$2$\\
&$U\oplus E_7\oplus A_2$ &$1$ \\
&$U\oplus D_9$ &$1$ \\
&$U\oplus W_9$ &$1$\\
\end{tabular}
\quad

\vspace{10pt}
\begin{tabular}{ c | c } 
  Rank & \# Borcherds\\ & lattices \\
  \hline
$10$ &$13$\\
$9$ &$15$\\
$8$ &$19$\\
$7$ &$21$\\

\end{tabular}
\quad 
\begin{tabular}{ c | c } 
  Rank & \# Borcherds\\ & lattices \\
  \hline
$6$ &$28$\\
$5$ &$27$\\
$4$ &$24$\\
$3$ &$18$\\

\end{tabular}
\caption{Borcherds lattices of rank $ \ge 11$. The last column indicates the rank of the abelian subgroup of finite index in $\Aut(\mathcal{D}_L)$. For the complete list of Borcherds lattices, we refer to the ancillary file.}
\label{tab:Borcherds.lattices.11}
\end{table}

%$$W_9=\begin{pmatrix}
% -2 &1 &-1 &-1 &1 &-1 &1 &0 &-1\\
% 1 &-2 &1 &1 &0 &0 &0 &0 &0\\
% -1 &1 &-2 &-1 &1 &-1 &1 &0 &-1\\
% -1 &1 &-1 &-2 &0 &0 &1 &0 &-1\\
% 1 &0 &1 &0 &-2 &1 &-1 &0 &1\\
% -1 &0 &-1 &0 &1 &-2 &1 &0 &-1\\
% 1 &0 &1 &1 &-1 &1 &-2 &0 &1\\
% 0 &0 &0 &0 &0 &0 &0 &-2 &1\\
% -1 &0 &-1 &-1 &1 &-1 &1 &1 &-4 
%\end{pmatrix}.$$

\normalsize

\printbibliography
\end{document}